\documentclass[11pt]{amsart}
\usepackage[utf8]{inputenc}
\usepackage[T1]{fontenc}

\usepackage{hyperref}
\usepackage[dvipsnames]{xcolor}
\usepackage{amsmath,amsfonts,amssymb,amsthm,epsfig,epstopdf,titling,url,array,mathtools,faktor,paralist, xfrac, mathrsfs,pifont,yfonts,mathabx,leftindex,adjustbox,stmaryrd, mathdots,booktabs, caption}
\usepackage[new]{old-arrows}
\usepackage{ytableau}
\usepackage[super]{nth}
\usepackage[foot]{amsaddr}
\usepackage{graphicx}
\usepackage{amsaddr}

\numberwithin{figure}{section}
\numberwithin{table}{section}
\newtheorem{theorem}{Theorem}[section]

\newtheorem{thmdef}[theorem]{Theorem/Definition}
\newtheorem{lemma}[theorem]{Lemma}
\newtheorem{prop}[theorem]{Proposition}

\theoremstyle{definition}
\newtheorem{definition}[theorem]{Definition}
\newtheorem{example}[theorem]{Example}

\theoremstyle{remark}
\newtheorem{remark}[theorem]{Remark}

\numberwithin{equation}{section}
\newfont{\tap}{tap scaled 650}

\def \N{{\mathbb N}}

\def \FF{F} 

\def \Z{{\mathbb Z}}

\def\F{\mathcal{F}}

\def\KK{{\mathbb K}}

\def \cA{{\mathcal A}}

\definecolor{dgreen}{rgb}{0,0.5,0}        
\definecolor{dred}{rgb}{0.5,0,0}

\DeclareMathOperator{\Hom}{Hom_{\mathbb{K}Q}}
\DeclareMathOperator{\Ext}{Ext_{\mathbb{K}Q}^{1}}

\DeclareMathOperator{\rep}{rep}

\newcommand{\Gr}{\operatorname{Gr}}

\newcommand{\db}{{\mathcal D}^b}

\renewcommand\mod{{\operatorname{\,-\,mod}}}
\DeclareMathOperator{\rmod}{\operatorname{mod\,-\,\mathcal{B}}}

\newcommand\restr[2]{{
  \left.\kern-\nulldelimiterspace 
  #1 
  \littletaller 
  \right|_{#2} 
  }}

\newcommand{\littletaller}{\mathchoice{\vphantom{\big|}}{}{}{}}

\begingroup\lccode`~=`& \lowercase{\endgroup
\newcommand{\spmat}[1]{%
  \left[
  \let~=&
  \begin{smallmatrix}#1\end{smallmatrix}
  \right]
}}

\definecolor{bettergreen}{RGB}{0,125,0}
\definecolor{alizarin}{rgb}{0.85, 0.15, 0.26}
\definecolor{azure}{rgb}{0.0, 0.5, 1.0}
\definecolor{upmaroon}{rgb}{0.70, 0.07, 0.07}

\usepackage{tikz}
\usepackage{tikz-cd}
\usetikzlibrary{arrows,positioning,scopes,trees,matrix,fit,backgrounds,shapes,fit,calc}

 \usepackage{hyperref}
 \hypersetup
 {
     colorlinks=true,
     linktocpage=true,
     linkcolor=upmaroon,
     filecolor=magenta,
     citecolor=bettergreen,
     urlcolor=azure,
     pdftitle={Growth From Representation Theory}
 }

\makeatother

\title{Growth of infinite frieze patterns of affine type}

\author{Karin Baur}
\address{Ruhr-Universit\"at Bochum, Fakult\"at f\"ur Mathematik, Universit\"atsstrasse 150, D-44801 Bochum, Germany}
\email{ka.baur@me.com}

\author{Anna Felikson}
\address{Department of Mathematical Sciences, Durham University, Mathematical Sciences \& Computer Science Building, Upper Mountjoy Campus, Stockton Road, Durham, DH1 3LE, UK}
\email{anna.felikson@durham.ac.uk}

\author{Deepanshu Prasad}
\email{deepanshu.prasad@gmail.com}

\author{Pavel Tumarkin}
\email{pavel.tumarkin@durham.ac.uk}

\author{Emine Y\i ld\i r\i m}
\address{International Center for Mathematical Sciences, Institute of Mathematics and
Informatics, Bulgarian Academy of Sciences, Acad. G. Bonchev Str., Bl. 8, Sofia
1113, Bulgaria}
\email{e.yildirim@math.bas.bg}


\begin{document}

\begin{abstract}
We analyse the growth coefficients of infinite frieze patterns arising from cluster algebras 
using cluster modular groups and  cluster categories.  For a fixed cluster category of affine type, we prove that the collection of infinite frieze patterns given by both the homogeneous and non-homogeneous stable tubes all have the same  growth coefficients.  We also derive and verify an explicit formula for the \(k\)-th growth coefficient, expressed directly in terms of data from homogeneous tubes, or, alternatively, from appropriate elements of the corresponding cluster algebra. 
\end{abstract}

\maketitle

\setcounter{tocdepth}{1}
\tableofcontents

\section{Introduction}

The combinatorial notion of a \emph{frieze pattern} goes back to Coxeter's work in the early 1970s~\cite{C71}. A frieze pattern is an infinite strip of integers formed by finitely or infinitely many horizontal rows, bounded above (and possibly below) by rows of $0$'s and with a row of $1$'s, see Figure~\ref{fig:x_ij}, and satisfying the local determinant or ``diamond'' relation: for each four entries arranged as
\[
    \begin{tikzcd}[cramped,sep=0.001em]
	   & a \\[-0.5em]
	   b && c \\[-0.5em]
	   & d
    \end{tikzcd}
    \qquad\text{one has}\qquad bc-ad=1.
\]
If the pattern is also bounded below, it is called \textit{finite}, otherwise, it is \textit{infinite}. For finite frieze patterns, 
Coxeter proved fundamental properties such as periodicity and a glide symmetry, and Conway and Coxeter classified them in terms of polygon triangulations~\cite{CC73a,CC73b}. Because of this origin, the finite patterns are often called \emph{Conway–Coxeter frieze patterns} (or \emph{Conway–Coxeter friezes}).

Since those early results the notion of a frieze pattern has been expanded considerably and proved to be a unifying combinatorial thread across many areas of modern mathematics.  Rotated by $45^\circ$, Conway–Coxeter frieze patterns can be viewed as integer-valued functions on a grid that satisfy this determinantal relation on every $2\times2$ square, a viewpoint that led Assem, Reutenauer and Smith to the notion of $\mathrm{SL}_2$-tilings of the plane~\cite{ARS}.  Higher-rank analogues, including $\mathrm{SL}_k$-tilings and $\mathrm{SL}_k$-frieze patterns, were first introduced in~\cite{CR72} and developed in~\cite{BR,FK13,MGOST14,BFGST21,PS25}.  Furthermore, the theory has been extended to frieze patterns with coefficients~\cite{CHJ, SHL23}, over algebraic numbers~\cite{CHP}, and into a noncommutative setting~\cite{CHJ1, CHJ2}.

Frieze patterns arise from cluster algebras, as pointed out in~\cite{CC06} for cluster algebras of Dynkin type $A$. The seminal work of Fomin, Shapiro, and Thurston~\cite{FST08} established surface cluster algebras, giving a bijection between seeds and (tagged) triangulations of surfaces. Fomin and Thurston~\cite{FT18} further showed that cluster variables can be realized as Penner's lambda-lengths for arcs in the decorated Teichm\"uller space, which satisfy the Ptolemy relation. 
Since then, finite and infinite frieze patterns have been associated to various types of surfaces, see for example~\cite{AB25,BM2009} in type $D$ or~\cite{BPT16, BC21} for infinite frieze patterns from 
annuli and more generally, the characterisation of positive integral friezes (as homomorphisms) arising from bordered surfaces,~\cite{CFGT21,FT25}.

The study of frieze patterns has been profoundly advanced through their connection to the representation theory of quivers and 2-Calabi-Yau triangulated categories. A quiver is an oriented graph 
$Q=(Q_{0},Q_1)$, where $Q_0$ are the vertices, $Q_1$ the arrows (oriented edges). If the underlying oriented graph is an extended Dynkin diagram of type $\widetilde{A}, \widetilde{D}$ or $\widetilde{E}$, we say that $Q$ is an 
{\em affine quiver}. A foundational result, conjectured by Assem-Reutenauer-Smith~\cite{ARS} and proven by Keller-Scherotzke~\cite{KS11}, establishes that 
frieze patterns 
satisfy linear recursions precisely when the underlying quiver is of Dynkin or affine (Euclidean) type. This result, deeply rooted in Auslander--Reiten theory, resolved a fundamental question and spurred further investigation into frieze patterns beyond these classical types, leading to significant developments in a broader categorical setting~\cite{GS,HJ1, HJ2, KPQ}. The natural habitat for these generalizations is the framework of 2-Calabi-Yau cluster categories, see a survey article~\cite{F25}. 
For an acyclic quiver \( Q \) 
whose underlying graph is an 
extended Dynkin diagram, 
the AR-quiver of the cluster category \( \mathcal{C}(Q) \) decomposes into a transjective component and infinitely many regular components. These regular components have been leveraged to construct infinite frieze patterns via the Caldero-Chapoton map (CC map), a technique widely employed in the literature~\cite{ARS, BBGTY}.

In this paper, we investigate the so-called growth coefficients of infinite frieze patterns by using the theory of cluster modular groups (Section~\ref{sec:cl-alg-mod-gr}) and representation theory of quivers (Section~\ref{sec:rep-thy-growth}). 
In particular, we consider frieze patterns given by the tubes arising from (cluster categories) of affine quivers. 

We obtain an explicit formula for the growth coefficient by using the double arrow configuration in the mutation class of an affine Dynkin quiver, see Propositions~\ref{g_a},~\ref{g_d}, Theorem~\ref{g_e} and Proposition~\ref{prop:g_rep}.
Furthermore, we show that all these frieze patterns for any given mutation equivalence class have the same growth coefficients, Theorem~\ref{UnderstandingHigherGrowthCoeff}, and can be understood using data from the homogeneous tube in the Auslander-Reiten quiver of the cluster category.

Independent of our work, in~\cite{PS26}, Plamondon and Stella define infinite friezes for all affine types (not necessarily simply laced) using root systems. They use the machinery of theta functions to show that friezes coming from the same root system share the same growth coefficients.

\subsection*{Acknowledgements}
A.F. and P.T. are grateful to Zack Greenberg for useful discussions. K.B. was supported by the EPSRC programme grant EPSRC (EP/W007509/1). 
K.B. acknowledges the support of the Centre for Advanced Study in Oslo, Norway that funded and hosted the research project ``Geometric Methods in Higher Homological Algebra'' in 2024/26. D.P. and E.Y. are grateful to Charles Paquette for useful discussions. D.P. and E.Y. were supported by the Bulgarian National Science Fund Contract No: KP-06-N92/5, the Ministry of Education and Science of the Republic of Bulgaria, grant DO1-239/10.12.2024 and the Simons Foundation grant SFI-MPS-T-Institutes-00007697.

%
\section{Frieze patterns and friezes as homomorphisms}\label{sec:friezes}

In this section, we review essential properties of infinite frieze patterns, and recall the definition of a \emph{frieze} as a homomorphism from cluster algebra to $\Z$.

We denote the entries in a frieze pattern  
by $x_{ij}$ where $i\in \Z$ and $j\in \Z_{\ge 0}$ as in Figure~\ref{fig:x_ij}, starting with $x_{i-1,i-1}$ for the row of $0$'s. The number of the row with the entries $x_{i_k,j_k}, k\in \Z$, is defined to be $j-i-1$, so that the row of $0$'s is in row $-1$, the row of $1$'s in row $0$, etc.

\begin{figure}[!ht]
\epsfig{file=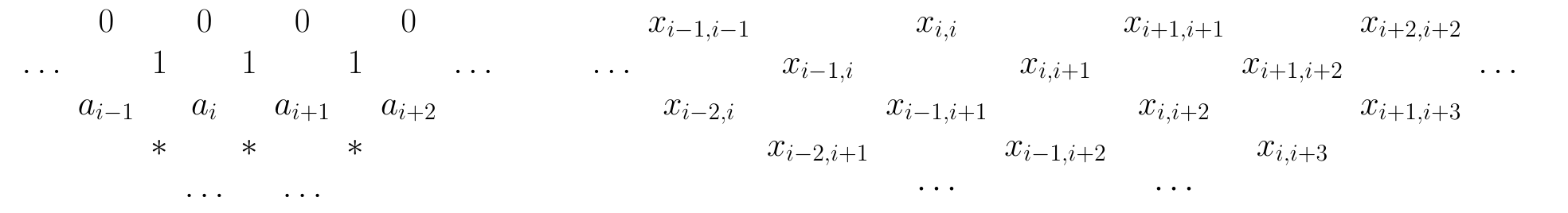,width=0.95\linewidth} 
\caption{Labeling of the entries in a frieze pattern.}  \label{fig:x_ij}
\end{figure}

By the diamond rule, any frieze pattern \( \mathcal{F} \) is determined by its first non-trivial row, namely the row with entries $a_i:=x_{i-2,i}$, i.e. by \( (a_i)_{i \in \mathbb{Z}} \). This row is called the \emph{quiddity row}. A frieze pattern is called \emph{periodic} if the quiddity row is periodic. In particular, if $\mathcal{F}$ is \( n \)-periodic, it is completely determined by an \( n \)-tuple $q = (a_i, \dotsc, a_{i+n-1})$ from its quiddity row. Every such tuple is called a \emph{quiddity sequence} of \( \mathcal{F} \).

Throughout this article, we focus on infinite periodic frieze patterns. 
All quiddity sequences are considered up to cyclic equivalence, i.e., \( (a_1, \dots, a_n) \sim (a_2, \dots, a_n, a_1) \). An example of an infinite frieze pattern is shown in Figure \ref{fig:infinite_frieze}. 

Let $\F$ be an $n$-periodic infinite frieze pattern. 
We consider the difference $x_{i,kn}-x_{i+1,kn-1}$ between entries in row $kn$ and row $(kn-1)$ which is the entry 
directly above row $kn$. The following proposition tells 
us that this difference is constant 
across the row, i.e.~independent of $i$ and that these differences are given by a recursive formula.

\begin{prop}(see ~\cite[Section 2]{BFPT19},~\cite[Proposition 4.1.1]{BCJKT})~\label{PropertiesofGrowthCoeff}
Let $\mathcal{F}$ be a $n$-periodic infinite frieze pattern. For every $k\in\mathbb{N}$, we have the following properties:
    \begin{enumerate}
        \item $x_{i,kn+i-1} - x_{i+1,kn+i-2}=s_k(\mathcal{F})$ for all $i\in\mathbb{Z}$; 
        \item $s_{k+1}(\mathcal{F})=s_1(\mathcal{F})s_{k}(\mathcal{F})-s_{k-1}(\mathcal{F})$, where $s_{0}(\mathcal{F})=2$
    \end{enumerate}
\end{prop}

We call the difference 
$s_k(\F):=x_{1,kn}-x_{2,kn-1}$, for $k\in \N$, the \emph{$k$-th growth coefficient of $\F$} and $s_{1}(\mathcal{F})$ its \emph{principal growth coefficient}. By the proposition, this is an invariant of the frieze pattern. 

Recall that the 
\emph{normalized Chebyshev polynomials of the first kind} are defined inductively by:
\[
    T_0(x) = 2,\quad T_1(x) = x,\quad \text{and} \quad T_{k+1}(x) = x T_k(x) - T_{k-1}(x) \quad \text{for any } k \geq 1
\]
and 
\emph{normalized Chebyshev polynomials of the second kind} are defined inductively by:
\[
    S_0(x) = 1,\quad S_1(x) = x,\quad \text{and} \quad S_{k+1}(x) = x S_k(x) - S_{k-1}(x) \quad \text{for any } k \geq 1
\]
Moreover, $T_{k}(x)=S_{k}(x)-S_{k-2}(x)$, where we set $S_{-2}(x)=-1$ and 
$S_{-1}(x)=0$.

\begin{remark}\label{rem:Cheby1-2}

If we set $x=s_1(\F)$ we obtain 
$s_k(\F)=T_k(x)$ for all $k\ge 0$, so the growth coefficients 
$s_k(\F)$ satisfy the normalized Chebyshev polynomials of the first kind.
\end{remark}

\begin{figure}
\centering
\[
\begin{tikzcd}[row sep=1.5ex, column sep=1.5ex]
    & \dotsb && 0 && 0 && 0 && 0 & {} & \dotsb \\
    \dotsb && 1 && 1 && 1 && 1 && 1 && \dotsb \\
    & \dotsb && 8 && 2 && 8 && 2 && \dotsb \\
    \dotsb && 15 && 15 && 15 && 15 && 15 && \dotsb \\
    & \dotsb && 28 && 112 && 28 && 112 && \dotsb \\
    && \vdots && \vdots && \vdots && \vdots && \vdots
\end{tikzcd}
\]
\caption{A $2$-periodic infinite frieze pattern with quiddity sequence $(8,2)$.}
\label{fig:infinite_frieze}
\end{figure}

\begin{remark}\label{rem:types}
As described in~\cite{BCJKT, BBGTY}, for a given affine quiver $Q$, one can define an infinite frieze pattern from its non-homogeneous stable tubes in the associated 
cluster category $\mathcal{C}(Q)$. 
Fix a cluster tilting object 
$T$ of $\mathcal{C}(Q)$. 
This corresponds to an initial cluster of the associated cluster algebra. 
We apply the 
Caldero--Chapoton map or CC-map $X_?^T$, see 
Section~\ref{sec:tame-hered-alg} to the $\tau$-orbit of the quasi-simple objects $(R_1,R_2,\dots, R_n)$ of a 
non-homogeneous tube of rank $n$, thus 
obtaining $n$ cluster variables. We then specialise all variables of the initial cluster (corresponding to $T$) to $1$ and obtain a sequence $(x^{T}_{R_1}, \dots, x^{T}_{R_{n}})$ of integers. 

This sequence serves as the quiddity row, and the diamond rule is applied recursively to generate an infinite frieze pattern. 
For an affine quiver of type $\widetilde{A},\widetilde{D}$, or $\widetilde{E}$, the infinite frieze patterns constructed in this manner will be called \emph{infinite frieze patterns of affine type 
$A,D$ or $E$}, respectively. 
These infinite frieze patterns are $n$-periodic by 
construction. 
\end{remark}

\begin{remark}\label{rem:frieze-surface}
Another way to define infinite frieze patterns is using surfaces, similarly as for 
Conway-Coxeter friezes. A {\em (marked) surface} is a connected oriented surface with boundary components and a finite set of marked points on the boundary or in the interior. The marked points in the interior are called punctures. Non self-crossing curves with endpoints at the marked points are called arcs (curves are always considered up to homotopy fixing endpoints).  
A triangulation $T$ of the surface is a maximal collection of pairwise non-crossing arcs. Any triangulation defines a cluster or the surface cluster algebra,~\cite{FST08}. 
Then one obtains an infinite frieze pattern for every connected component of the boundary by counting matching numbers between arcs at that component and triangles of the triangulation, see Section~\ref{sec:annuli} below. 
\end{remark}

We also recall the following notion: 
\begin{definition}\label{def:frieze-homom}
Let $S$ be a marked surface and $\cA= \cA(S)$ its surface cluster algebra. 
A \textit{frieze $\FF$ on $S$} 
is a homomorphism from $A$ to $\mathbb{Z}$ such that the image of every cluster variable (i.e. of every non-self-crossing arc in the surface) is positive. 
\end{definition}

It follows from~\cite[Sections 3,4]{MSW13} that $\FF$ also assigns positive integers to all self-crossing curves in $S$ with endpoints at marked points, as well as to all non-contractible closed curves. 

\begin{remark}
    \label{hyp}
    Due to the results of~\cite{Pen,FT18}, a frieze $\FF$ on a punctured surface $S$ can be also understood as a decorated hyperbolic structure on $S$. In this interpretation, the value of (cluster variable of) an arc $\gamma$ is equal to the $\lambda$-length of $\gamma$, and the value of an element of the cluster algebra corresponding to a simple closed curve $\sigma$ is equal to $2\cosh\frac{l(\sigma)}{2}$, where $l(\sigma)$ is the hyperbolic length. See~\cite{FT25} for details.   
\end{remark}

We note that friezes $\FF$ on surfaces also give rise to (finite or infinite) frieze patterns. 
If the surface is a polygon, one obtains a Conway--Coxeter frieze pattern from $\FF$ by arranging the values under $\FF$ of the (cluster variables of the) diagonals $ij$ as  the entries $x_{ij}$ in the pattern. Friezes $\FF$ on annuli or punctured discs give infinite periodic frieze patterns by taking the values of the (cluster variables of the) arcs homotopic to a concatenation of $k\ge 2$ boundary segments of a component of the boundary to be the entries 
$x_{i,i+k}$, see Section~\ref{sec:annuli}, i.e. 
the quiddity sequence, or by 
considering the values of certain arcs between 
the two punctures (\cite[Section 3]{BBGTY}). 

In general, if the surface is not a polygon, 
since a frieze $\FF$ associates an integer to \textit{every} cluster variable, the notion of a frieze as in  Definition~\ref{def:frieze-homom} 
is different from the (collection of) infinite 
frieze patterns arising from arcs at boundary components or from arcs between punctures (similarly as above).

\begin{remark}
\label{rem:all}
It was shown in~\cite{BPT16} that every infinite  $n$-periodic integer frieze pattern $\F$ originates from a triangulation of a marked surface $S$, where $S$ is either an annulus $A_{m,n}$ (for some $m$) or a punctured disc $D_n$ with $n$ boundary marked points. Namely, let $(a_1,\dots, a_n)$ be the quiddity sequence of $\F$. Let 
$A_{m,n}$ or $D_n$ be the corresponding surface and denote the marked points on (one component of) the boundary by $P_1,\dots, P_n$, clockwise. Then there exists a triangulation $T$ of this surface such that every $a_i$ is equal 
to the number of triangles incident to  
vertex $P_i$ in a neighbourhood of $P_i$ (some of the triangles might be counted twice).

We note that in the case of a punctured disc $D_n$, the entries in any diagonal of $\F$ form $n$ arithmetic progressions (Proposition 3.11 in~\cite{Tsch15}) and the growth coefficients are 
all equal to $2$. 
All other infinite periodic frieze patterns have $s_1(\F)>2$ (Proposition 4.4~\cite{BFPT19}) and can be constructed via triangulations of annuli. We will from now on concentrate on the infinite frieze patterns whose entries grow fast, i.e. which arise from triangulations of annuli. 
\end{remark}

\section{Growth via Cluster Modular Group}
\label{sec:cl-alg-mod-gr}

In this section, we give a combinatorial interpretation of the growth coefficients of infinite frieze patterns of types $\widetilde A, \widetilde D, \widetilde E$.

\subsection{Growth and Skein Relations}
\label{sec:skein}

We start by considering types $\widetilde A$ and $\widetilde D$ by using their surface realisations. 

\subsubsection{Affine $A$: annuli}\label{sec:annuli}

Let  $\F$ be an infinite $n$-periodic frieze pattern with entries $\{x_{i,j}\}$, 
let $T$ be a 
triangulation of an annulus $A_{m,n}$ giving rise to $\F$ (by Remark~\ref{rem:all}, we may assume that $S=A_{m,n}$ is an annulus). We 
denote by $\FF$ the corresponding frieze on the annulus $A_{m,n}$.

By Proposition~\ref{PropertiesofGrowthCoeff}, the first growth coefficient (or simply {\em growth}) is equal to $x_{i,i+n+1}-x_{i+1,i+n}$, 
independently of $i$. 
The entries in $\F$ are given by matching numbers between the triangles of $T$ and the vertices on the boundary: to compute $x_{ij}$ for $j-i>1$, one takes an arc $\gamma_{ij}$ in $S$ starting at vertex $P_{i}$ and ending at $P_{j}$ 
which is homotopic to the concatenation of $j-i$ boundary segments, clockwise around the boundary: 
the entry $x_{ij}$ is equal to the number of matchings between triangles of $T$ and the vertices of $S$ covered by this arc, see Section 5 in~\cite{BPT16}. 
We call arcs arising from concatenating segments of the same boundary \textit{peripheral}.
We note that the arc $\gamma_{ij}$ is self-intersecting if and only if $j-i>n$, the first case of a self-intersecting arc being precisely the one needed to compute $x_{i,i+n+1}$.

Let $\sigma$ be the unique simple closed curve in the annulus. Denote by $\Theta=\FF(\sigma)$ the value of the element in the cluster algebra corresponding to $\sigma$ in the frieze $\FF$.

If we smooth the self-crossing of the arc $\gamma_{i,i+n+1}$ using skein relations (see~\cite[Section 2]{MSW}), we obtain the equation 
$x_{i,i+n+1}=x_{i+1,i+n}+\Theta\cdot 1$. 
Therefore, the growth 
coefficient of the frieze pattern $\F$ is equal to the value $\Theta$ (see also~\cite[Remark 6.5]{BMBFT}).  

This argument gives an alternative proof of the fact that the growth coefficients of the two infinite periodic frieze patterns given by the matching numbers at the two boundary components of the same triangulated annuls coincide,~\cite{BFPT19}.  

\begin{remark}
\label{br_a}
The same argument applied to one self-crossing of the arc $\gamma_{i,i+kn+1}$ shows that the $k$-th growth coefficient of $\F$ is equal to the value of $\FF$ on the $k$-th \emph{bracelet} (see~\cite[Definition 3.16]{MSW13}), i.e. on the closed curve going $k$ times around the annulus.    
\end{remark}

\subsubsection{Affine $D$: twice punctured discs}\label{sec:affineD}

A frieze $\FF$ arising from a triangulation of 
a twice punctured disc with $n$ boundary marked points contains three infinite periodic frieze patterns: one consists of values of all peripheral arcs (as in Section~\ref{sec:annuli}), arising from the matching numbers of the $\gamma_{ij}$. 
The other two arise from arcs at the two punctures as we recall below. 

Let $\sigma$ be the unique simple closed curve enclosing the two punctures and denote by $\Theta$ the associated integer as before. 
Similarly to the case $A_{m,n}$, by smoothing arcs $\gamma_{i,i+n+1}$, one can see that 
the growth of the infinite frieze pattern arising from the peripheral arcs is equal to $\Theta$.

\begin{remark}
\label{br-d}
Similarly to the case $\widetilde A$, the $k$-th growth coefficient is again equal to the value of $\FF$ on the $k$-th bracelet going $k$ times along the boundary of the disc.  
\end{remark}

To describe the quiddity sequences and growth of the other two infinite frieze patterns we need to 
recall the notion of \emph{tagged arcs} from~\cite{FST08}. 

If the surface $S$ has punctures, arcs are equipped with a \emph{tag} at their endpoints: at the boundary, the tag is always plain. At a puncture, a tag can be plain or notched. Two arcs whose untagged versions are different are said to be crossing at the  puncture $P$ if their tag at $P$ is different. Two arcs whose untagged versions coincide are non-crossing if and only if they have opposite tags at exactly one end.

Let $A$ and $B$ be the two punctures, let $\gamma$ be the arc $AB$ tagged plain at both ends. We denote by $\gamma_A$, $\gamma_B$ 
the arc $\gamma$ tagged notched only at $A$, $B$, respectively, and by $\gamma_{AB}$ the arc notched at both $A$ and $B$. Then the two remaining quiddity sequences arise from the values of of these four arcs in the frieze $\FF$
one infinite frieze pattern has as its quiddity sequence the values of arcs $\gamma_A$ and $\gamma_B$, while the other has as its quiddity the values of arcs $\gamma_{AB}$ and $\gamma$.

Consider the former one. The arcs $\gamma_A$ and $\gamma_B$ cross at $A$ and at $B$. Applying skein relations~\cite[Figure 23, Section 8]{MSW}, one obtains four closed 
curves from smoothening the pair $\gamma_A,\gamma_B$, and so the product $\FF(\gamma_A)\FF(\gamma_B)$ is 
equal to the sum of the corresponding four terms: $\Theta$, the values of the two loops around $A$ and $B$, and the value of the contractible loop.

The contractible loop has value $-2$ in the cluster algebra, while loops around single punctures have value $2$.  Therefore, $\FF(\gamma_A)\cdot \FF(\gamma_B)=\Theta+2$. The diamond rule implies that every entry of the second non-trivial row is equal to  $\FF(\gamma_A)\cdot \FF(\gamma_B)-1$, and thus it is equal to $\Theta+1$. Since the infinite frieze pattern is $2$-periodic, the growth is equal to the difference $(\Theta+1)-1=\Theta$ and so it coincides with the growth of the infinite frieze pattern consisting of the values of peripheral arcs.

The consideration of the latter case (with $\gamma$ and $\gamma_{AB}$) is very similar. Application of the skein relations gives precisely the same expression, which implies that the growth is also equal to $\Theta$.

In particular, this argument gives another proof of the result of~\cite{BBGTY} stating that the growths of the three infinite periodic frieze patterns obtained from $\widetilde D_n$ coincide.

\begin{remark}
The infinite frieze pattern 
arising from $\gamma_A$ and $\gamma_B$ 
can be understood as follows. 
Denote by $l_A$ the loop around $A$ based at $B$, and by $l_B$ a loop around $B$ based at $A$. Let $g:=\FF(\gamma)$ be the value of the untagged arc $\gamma$. Then $\FF(\gamma_A)=\FF(l_A)/g$, and $\FF(\gamma_B)=\FF(l_B)/g$ (see~\cite[Lemma 7.10]{FT18}). For $k>0$, denote by $Br_k$ the curve with endpoints $A$, $B$ going around $A$ and $B$ precisely $k-1$ times (e.g., $\gamma=Br_1$). Denote by $L_{k,A}$ and $L_{k,B}$ the loops based at $B$ and $A$ respectively going around the other puncture and then around both $A$ and $B$ precisely $k-1$ times (e.g., $L_{1,A}=l_A$).

Note that $L_{0,A}$ is a contractible loop at $B$ and thus $\FF(L_{0,A})=0$, and similarly  $\FF(L_{0,B})=0$. Note also that $B_1=\gamma$, so $Br_1/g=1$. We can describe the infinite periodic frieze pattern (arising from $\gamma_A$ and $\gamma_B$) 
with this notation: 
every $2k$-th row consist of alternating values $\FF(L_{k,A})/g$ and $\FF(L_{k,B})/g$, while every odd $(2k+1)$-st row consists of the value $\FF(Br_k)/g$.

The construction above can be also modeled as follows. We cut the surface along $\gamma$ to get an annulus with a boundary component containing two marked points $A$ and $B$, and the two boundary arcs are labeled by $\FF(\gamma)=g$. Restricting $\FF$ to the new surface we get a {\em frieze with coefficients} (we keep the same notation $\FF$). Then the infinite frieze pattern  consists of values of all the peripheral arcs of this new boundary component (tagged plain) divided by $g$.   

We can alter the construction by considering all peripheral arcs tagged notched on both ends, and substituting $g=\FF(\gamma)$ with $\FF(\gamma_{AB})$ in the denominator. The resulting infinite frieze pattern differs by a horizontal shift by one step.

\medskip

The construction of the frieze pattern arising from $\gamma$ and $\gamma_{AB}$ is very similar to the former one: 
its $2k$-th rows consist of alternating values $\FF(L_{k,A})/\FF(\gamma_A)$ and $\FF(L_{k,B})\cdot \FF(\gamma_A)$, while the $(2k+1)$-st rows still consist of the values $\FF(Br_k)/g$. Equivalently, we may cut the surface as above, but with the boundary arcs labeled by $\FF(\gamma_A)$. Then the infinite frieze pattern consists of the values of 
the peripheral arcs (at the new boundary) tagged notched at $A$ and plain at $B$ divided by $\FF(\gamma_A)$.

Note that in the latter case we could swap punctures $A$ and $B$ in the construction. The resulting infinite frieze pattern would differ by a horizontal shift by one step. 
\end{remark}

\subsection{Growth and Cluster Modular group}

The growth of infinite frieze patterns from cluster algebras in the affine types $\widetilde{A}$, $\widetilde{D}$  and $\widetilde{E}$ can be understood in terms of the \emph{cluster modular group}, which is the group of the mutation sequences in a cluster algebra preserving the given quiver. Cluster modular groups of affine cluster algebras were computed in~\cite{ASS}, we will use a different presentation by Greenberg and Kaufman~\cite{KG}. 

We will first consider the types $\widetilde{A}$ and $\widetilde{D}$, where the cluster modular group can be expressed via the mapping class group of the corresponding surface, and then extend this approach to the type $\widetilde{E}$.

%
\subsubsection{
Cluster algebras of type $\widetilde A_{m,n}$ and $\widetilde D_{n+2}$}
\label{ad}

We now consider the mapping class groups of the underlying surfaces of cluster algebras of type $\widetilde{A}_{m,n}$ and $\widetilde{D}_{n+2}$, namely, an annulus with $m$ points on one boundary component and $n$ points on the other one, and a twice punctured disc with $n$ points on the boundary, respectively. 

Both mapping class groups are virtually infinite cyclic, with the generator $\tau$ of the cyclic subgroup being the Dehn twist along the unique simple closed curve. By~\cite{FST08}, the cluster modular group is an extension of the mapping class group by $\Z_2$ at every puncture, so it is still virtually cyclic with the generator of its cyclic subgroup being the same Dehn twist $\tau$. According to Section~\ref{sec:skein}, the growth 
of any corresponding infinite frieze pattern $\F$ is equal to the value of $\FF$ on the element of the cluster algebra corresponding to the simple closed curve. In other words, the growth
can be interpreted as the value of $\FF$ on the element in the cluster algebra corresponding to the generator of an infinite cyclic subgroup of the cluster modular group.

This value can be easily read off from any seed of the cluster algebra with a quiver containing a double arrow. In this seed, the generator $\tau$ corresponds to a mutation in an endpoint of the double arrow (see~\cite[Lemma B.3]{KG}). In  case of $\widetilde A_{m,n}$, every such quiver looks like the one shown in Fig~\ref{fig-amn}(a), and it contains a subquiver shown in Fig.~\ref{fig-amn}(b) (we denote by $x_v$ the cluster variable corresponding to vertex $v$; the variables $x_a,x_b$ may be frozen, they are equal to $1$ in that case). Calculating the element of the cluster algebra corresponding to the unique simple closed curve, we obtain the explicit expression for the growth.
\begin{prop}
\label{g_a}
In the notation of Fig.~\ref{fig-amn}, the growth $\Theta$ of any infinite frieze pattern of type $\widetilde A_{m,n}$ is equal to 
    $$\Theta=\frac{x_1}{x_0}+\frac{x_0}{x_1}+\frac{x_ax_b}{x_0x_1}.$$
    \end{prop}
 Note that the value $\Theta$ does not depend on the seed as the element is defined in invariant way.

\begin{figure}[!ht]
\epsfig{file=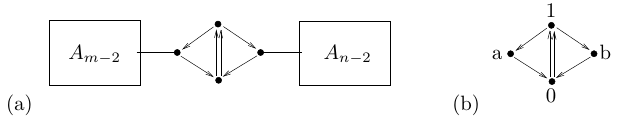,width=0.75\linewidth}\caption{Quiver of type $\widetilde A_{m,n}$ with a double arrow, and its subquiver.}\label{fig-amn}
\end{figure}

Similarly, in case of $\widetilde D_{n+2}$, every such quiver looks like the one shown in Fig~\ref{fig-dn}(a), and it contains a subquiver shown in Fig.~\ref{fig-dn}(b). Then one can calculate that the following holds. 
\begin{prop}
In the notation of Fig.~\ref{fig-dn}, the growth $\Theta$ of any infinite frieze pattern of type $\widetilde D_{n+2}$ is equal to
\label{g_d}
$$\Theta=\frac{x_1}{x_0}+\frac{x_0}{x_1}+\frac{x_ax_bx_c}{x_0x_1}.$$
\end{prop}
 As in the case $\widetilde A_{m,n}$, $\Theta$ does not depend on the seed as the element is defined invariantly.
\begin{figure}[!ht]
\epsfig{file=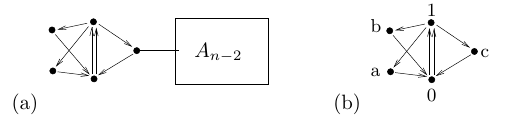,width=0.7\linewidth}  
\caption{Quiver of type $\widetilde D_{n+2}$ with a double arrow, and its subquiver.}
\label{fig-dn}
\end{figure}
\medskip

\subsubsection{Affine $E$-types}

The cluster modular groups of cluster algebras of types 
$\widetilde{E}_n$, for $n=6,7,8$,  are virtually infinite cyclic~\cite{FSTT}, we will use the explicit presentations 
from~\cite{KG}. 

In the notation of Fig.~\ref{fig-e}, the group is generated by the elements $\tau_a$, $\tau_b$, $\tau_c$ and (in the case $\widetilde{E}_6$) $\gamma$, where
\begin{itemize}
\item $\tau_a=\mu_a\mu_0\mu_1$;
\item $\tau_b=\mu_{b_1}\mu_b\mu_0\mu_1$;
\item $\tau_c=\mu_{c_k}\dots\mu_{c_1}\mu_c\mu_0\mu_1$, where $k=1,2,3$ for 
$\widetilde{E}_6$, $\widetilde{E}_7$ and 
$\widetilde{E}_8$, respectively (i.e., $k=n-5$); 
\item $\gamma$ is the unique symmetry of the quiver.
\end{itemize}

The mutation sequences should be followed with appropriate permutations of vertices of the quiver to preserve the quiver.

The relations between the generators are given by 
$$\tau_a^2=\tau_b^3=\tau_c^{k+2}$$
(where all the expressions above are equal to $\tau$), and also for $\widetilde{E}_6$ we have $\gamma\tau_a=\tau_a\gamma$, $\gamma\tau_b=\tau_c\gamma$.

\begin{figure}[!ht]
\epsfig{file=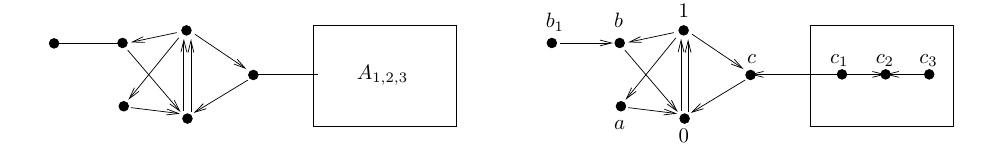,width=0.99\linewidth}  
\caption{Quiver $Q$ of type 
$\widetilde{E}_n$; the number of vertices $c_i$ is $k=n-5$.} 
\label{fig-e}
\end{figure}

Every quiver $Q$ with a double arrow in the mutation class 
of $\widetilde{E}_n$, $n=6,7,8$, looks like one shown in Fig~\ref{fig-e} (left), and it contains a subquiver shown in Fig.~\ref{fig-dn}(b). To such a $Q$, we define three subquivers $Q_a,Q_b,Q_c$ as follows.

$Q\setminus \{b,c\}$ consists of three connected components. Denote by $Q_a$ the subquiver of $Q$ obtained from $Q$ by removing all connected components of $Q\setminus \{b,c\}$ not containing $a$. Similarly, we can define  $Q_b$ and $Q_c$. More precisely,
$$Q_a=Q\setminus\{b_1,c_1,\dots,c_k\},\ Q_b=Q\setminus\{c_1.\dots,c_k\},\ Q_c=Q\setminus\{b_1\}.$$
Observe that $Q_a,Q_b,Q_c$ are of the type $\widetilde D_4$, $\widetilde D_5$, $\widetilde D_{k+4}$ respectively. 

We now claim the following.

\begin{lemma}
    The value
$$\Theta=\frac{x_1}{x_0}+\frac{x_0}{x_1}+\frac{x_ax_bx_c}{x_0x_1}$$
does not depend on the seed.
\end{lemma}

In other words, for any seed with quiver $Q$ the value of $\Theta$ is the same. 

\begin{proof}
Any two seeds with a given quiver are connected by an element of a cluster modular group. Thus, it is sufficient to show that any generator preserves $\Theta$. Observe that $\tau_a$ is an element of the cluster modular group of a cluster algebra of type $\widetilde D_4$ with quiver $Q_a$. By Proposition~\ref{g_d}, the value $\Theta$ is independent of the seed, and thus is invariant under the action of $\tau_a$. Similarly, $\tau_b$ and $\tau_c$ act on the subquivers $Q_b$ and $Q_c$ and thus do not change $\Theta$. In case of $\widetilde E_6$, the symmetry $\gamma$ exchanges vertices $b$ and $c$ and hence does not change $\Theta$ either.
    
\end{proof}

Now observe that for each of the three subquivers $Q_a,Q_b,Q_c$ we can consider the infinite periodic frieze pattern arising from arcs at the boundary 
of the twice punctured disc for $\widetilde D_4$,  $\widetilde D_5$, $\widetilde D_{k+4}$, respectively. 

These infinite frieze patterns are precisely the ones that can be constructed from $Q$ as in Remark~\ref{rem:types}, using the CC-map $X_?^T$. Note that categorification for the cluster algebra of $Q$ is compatible with the seed and mutation structure: seeds are cluster-tilting objects and mutation is respected by the CC-map $X_?^T$. By Proposition~\ref{g_d}, the growth of each of them is equal to $\Theta$.  We summarize the observation above as the following theorem.

\begin{theorem}
\label{g_e}
 Let $Q$ be any quiver of affine type $E$ with a double arrow. In the notation of Fig.~\ref{fig-e}, the growth $\Theta$ of any infinite frieze pattern 
 of type $\widetilde E_{678}$ is equal to
$$\Theta=\frac{x_1}{x_0}+\frac{x_0}{x_1}+\frac{x_ax_bx_c}{x_0x_1}.$$   
\end{theorem}

In the next section, we approach  Propositions~\ref{g_a} and \ref{g_d} and Theorem~\ref{g_e} using representation-theoretic methods 
and describe the growth coefficients using 
the data from the homogeneous tubes, 
see Proposition~\ref{prop:g_rep}.

%
\section{Growth via representation theory}\label{sec:rep-thy-growth}

Fix an affine quiver \(Q\) of type 
$A,D$ or $E$ with \(m\) vertices. Let $\mathbb{K}Q$ be the path algebra of $Q$. Denote by $\mathcal{C}(Q)$ the cluster category associated to $Q$, and by $T$ a basic cluster-tilting object \(T\in\mathrm{obj}(\mathcal{C}(Q))\). The $CC$-map with respect to $T$ denoted by $X^{T}_{?}$ (cf.~Section~\ref{sec:tame-hered-alg}). 
We will simply write $X_{?}$ when $T\cong\mathbb{K}Q$.

Let \(e\) be an extending vertex of \(Q\), and assume that \(e\) is a sink. Denote by \(P_e\) the indecomposable projective module at \(e\), and set $p := \underline{\dim}\,P_e$. We use $\delta$ to denote the dimension vector of the regular quasi-simple module in the homogeneous stable tube of the Auslander-Reiten quiver (AR quiver) $\Gamma(\mathbb{K}Q\mod)$. By Theorem~\ref{ParameterizationOfTubes}, every stable tube \((\mathcal{T},\tau)\) in the AR quiver $\Gamma(\mathbb{K}Q\mod)$ contains an indecomposable module \(M\in\Omega\) fitting into a short exact sequence
\[
0 \longrightarrow P_e \longrightarrow E \longrightarrow M \longrightarrow 0,
\]
where \(E\) is the unique (up to isomorphism) indecomposable preprojective module of dimension vector \(\delta + p\). Moreover, we have \(\langle p,\delta\rangle = 1\).

Now let \((\mathcal{T},\tau)\) be a non-homogeneous stable tube of rank \(n>1\). The mouth of \(\mathcal{T}\) contains a \(\tau\)-cycle of quasi-simple modules \((R_1,R_2,\dots,R_n)\). For each \(i\), define $\alpha_i := \underline{\dim}\,R_i$ 
and for $j\ge 0$, we write \(R_i[j]\) to denote the indecomposable regular module in \(\mathcal{T}\) with quasi-socle \(R_i\) and quasi-length \(j\). After relabeling if necessary, we may assume \(\alpha_n(e) \neq 0\). Denote by \(\mathscr{I}_1\) and \(\mathscr{I}_2\) the indecomposable preinjective module with
$\underline{\dim}\,\mathscr{I}_1 = \delta - p$, and $\underline{\dim}\,\mathscr{I}_2 = \alpha_n - p$.

Using \cite[Section 8]{Ding2013-xa} we establish the following identities among the elements of the cluster algebra, with respect to any cluster tilting object $T$ in $\mathcal{C}(Q)$.

\begin{prop}
\label{InterpretationOfGrowth}
Let $(\mathcal{T},\tau)$ be a non-homogeneous 
tube of rank $n$ of $\mathcal{C}(Q)$. 
With the above notation, the 
following hold: 
    \begin{enumerate}
        \item If $e$ is a sink, then $X^{T}_{R_{1}[n]}-X^{T}_{R_{2}[n-2]}=X^{T}_{M_{\lambda}}$
        \item If $e$ is a source, then $X^{T}_{R_{n}[n]}-X^{T}_{R_{1}[n-2]}=X^{T}_{M_{\lambda}}$
    \end{enumerate}
    where $M_{\lambda}$ is any quasi-simple in a stable homogeneous tube $(\mathcal{H}_{\lambda},\tau)$ for $\lambda\in\mathbb{P}^{1}$.
\end{prop}

\begin{proof}
    Let $T\in\mathrm{obj}(\mathcal{C}(Q))$ be a basic cluster-tilting object. First assume that $e$ is a sink. By \cite[section 8]{Ding2013-xa}, we get the following triangles in $\mathcal{C}(Q)$.
    \[
        P_{e}\rightarrow E\rightarrow R_{1}[n]\rightarrow \Sigma P_{e}
    \]
    \[
        R_{1}[n]\rightarrow R_{2}[n-1]\oplus \tau^{-1}\mathscr{I}_{2}\rightarrow P_{e}\rightarrow \Sigma (R_{1}[n])
    \]
    with $\dim \mathrm{Ext}^{1}_{\mathcal{C}(Q)}(R_1[n],P_{e})=1$. Hence, by Theorem \ref{multCK2}, we get
    \begin{align}
    \label{identity 1}
        X^{T}_{P_{e}}X^{T}_{R_{1}[n]} =X^{T}_{E} + X^{T}_{R_{2}[n-1]}X^{T}_{\tau^{-1}\mathscr{I}_{2}}
    \end{align}
    We also have
    \[
        \begin{tikzcd}
            R_2[n-1] \arrow[r] & \tau^{-1}\mathscr{I}_{1} \arrow[r] & \tau^{-1}\mathscr{I}_2 \arrow[r] & R_1[n-1]
        \end{tikzcd}
    \]
    \[
        \begin{tikzcd}
            \tau^{-1}\mathscr{I}_2 \arrow[r] & R_2[n-2] \oplus P_e \arrow[r] & R_2[n-1] \arrow[r] & \mathscr{I}_2
        \end{tikzcd}
    \]
    with $\dim\mathrm{Ext}_{\mathcal{C}(Q)}^1(\tau^{-1}\mathscr{I}_{2}, R_{2}[n-1])=1$. Hence, by using Theorem \ref{multCK2} we get that
    \begin{align}
    \label{identity 2}
        X^{T}_{R_{2}[n-1]}X^{T}_{\tau^{-1}\mathscr{I}_{2}}= X^{T}_{R_{2}[n-2]}X^{T}_{P_{e}}+X^{T}_{\tau^{-1}\mathscr{I}_{1}}
    \end{align}
    For any homogeneous stable tube $(\mathcal{H_{\lambda}},\tau)$ we get the triangles
    \[
        P_{e}\rightarrow E\rightarrow M_{\lambda}\rightarrow \Sigma P_{e}
    \]
    \[
        M_{\lambda}\rightarrow\tau^{-1}\mathscr{I}_{1}\rightarrow P_{e}\rightarrow M_{\lambda}
    \]
    with $\dim\mathrm{Ext}^{1}_{\mathcal{C}(Q)}(M_{\lambda},P_{e})=1$, independent of the choice $\lambda$. 
    Hence, by Theorem \ref{multCK2} we get
    \begin{align}
    \label{identity 3}                    X^{T}_{P_{e}}X^{T}_{M_{\lambda}}=X^{T}_{E}+X^{T}_{\tau^{-1}\mathscr{I}_{1}}
    \end{align}
    By using \eqref{identity 1}, \eqref{identity 2}, and \eqref{identity 3}, we get
    \[
        X^{T}_{P_{e}}X^{T}_{R_{1}[n]} = X^{T}_{P_{e}}X^{T}_{M_{\lambda}} + X^{T}_{R_{2}[n-2]}X^{T}_{P_{e}}
    \]
    Hence,
    \[
        X^{T}_{R_{1}[n]}-X^{T}_{R_{2}[n-2]}=X^{T}_{M_{\lambda}}.
    \]
    The proof of the second statement (i.e. for the case where $e$ is a source) follows by a similar argument, employing the triangles constructed in \cite[Section 8]{Ding2013-xa}. 

\end{proof}

\begin{remark}
\label{relatingBraceletsandMlambda}
    By Proposition \ref{InterpretationOfGrowth}, together with~\cite[Proposition 6.1]{GMV19} in 
    affine type $A$ and~\cite[Proposition 3.5]{BBGTY} in affine type $D$ along with the skein relations~\cite[Proposition 4.2]{MSW13}, we get that \( X(\text{Brac}_1) = X_{M_\lambda} \), where we write \(\text{Brac}_k\) to denote the $k$-bracelet as defined in~\cite[Definition 3.16]{MSW13}, and 
    where \(X(\text{Brac}_k)\) is the associated 
    cluster algebra element. This agrees with Remarks~\ref{br_a} and~\ref{br-d}. 
\end{remark}

Using Proposition~\ref{InterpretationOfGrowth} we can give an interpretation of the \(k\)-th growth coefficient of infinite frieze patterns of affine type as we show next.  

Let $(\mathcal{H}_{\lambda},\tau)$ be a homogeneous tube in $\Gamma(\mathbb{K}Q\mod)$, for $\lambda\in\mathbb{P}^{1}$. 
We write $M_{\lambda}$ to denote the quasi-simple at the mouth of $\mathcal{H}_{\lambda}$ and  $M_{\lambda}[m]$ to denote the indecomposable module in $\mathcal{H}_{\lambda}$ with quasi-socle $M_{\lambda}$ and quasi-length $m\in\mathbb{N}$. We adopt the convention that \( M_{\lambda}[0] \cong 0 \). Consequently, we set $X^{T}_{M_{\lambda}[0]} = 1$, and $X^{T}_{M_{\lambda}[-1]} = 0$ and $X^{T}_{M_{\lambda}[-2]}=-1$. Our goal is to understand the \(k\)-th growth coefficient via \( X^{T}_{M_{\lambda}[k]} \). Before doing so, we establish the following recurrence relation satisfied by \( X^{T}_{M_{\lambda}[k]} \).

\begin{lemma}
\label{ChebyshevRecurRelnForMlambda}
    The following recurrence relation holds in $(\mathcal{H}_{\lambda},\tau)$
    \[
        X^{T}_{M_{\lambda}[k+1]}=X^{T}_{M_{\lambda}}X^{T}_{M_{\lambda}[k]}-X^{T}_{M_{\lambda}[k-1]},\quad\text{for}\;k\in\mathbb{N}
    \]
\end{lemma}

\begin{proof}
    By using the fact that $\mathcal{H}_{\lambda}$ is a stable tube of rank $1$, and ~\cite[Theorem 2.2, Chapter X]{SS07}, we have the following almost-split sequence
    \[
        0\rightarrow M_{\lambda}[t-1]\rightarrow M_{\lambda}[t]\oplus M_{\lambda}[t-2]\rightarrow M_{\lambda}[t-1]\rightarrow 0\quad\text{for}\;t\geq2
    \]
    Therefore, we get the following relation for the $CC$-map
    \begin{equation}
    \label{XMlambda}
        X^{T}_{M_{\lambda}[t]}=\frac{X^{T}_{M_{\lambda}[t-1]}X^{T}_{M_{\lambda}[t-1]}-1}{X^{T}_{M_{\lambda}[t-2]}},\quad\text{for}\;t\geq2
    \end{equation}
    Now, to prove the recurrence relation we apply induction on $k$. For $k=1$, by using the relation (\ref{XMlambda}) for $t=2$, we get $X^{T}_{M_{\lambda}[2]}=X^{T}_{M_{\lambda}}X^{T}_{M_{\lambda}}-1=X^{T}_{M_{\lambda}}X^{T}_{M_{\lambda}}-X^{T}_{M_{\lambda}[0]}$. Assume the recurrence relation is true for $k$, i.e., $X^{T}_{M_{\lambda}[k]}=X^{T}_{M_{\lambda}}X^{T}_{M_{\lambda}[k-1]}-X^{T}_{M_{\lambda}[k-2]}$. We will now prove the recurrence relation for $k+1$. Using the relation (\ref{XMlambda}) for $t=k+1$, we get
    \begin{align*}
        X^{T}_{M_{\lambda}[k+1]}&=\frac{X^{T}_{M_{\lambda}[k]}X^{T}_{M_{\lambda}[k]}-1}{X^{T}_{M_{\lambda}[k-1]}}\\
        &=\frac{X^{T}_{M_{\lambda}[k]}(X^{T}_{M_{\lambda}}X^{T}_{M_{\lambda}[k-1]}-X^{T}_{M_{\lambda}[k-2]})-1}{X^{T}_{M_{\lambda}[k-1]}}\\
        &=\frac{X^{T}_{M_{\lambda}}X^{T}_{M_{\lambda}[k]}X^{T}_{M_{\lambda}[k-1]}-(X^{T}_{M_{\lambda}[k]}X^{T}_{M_{\lambda}[k-2]}+1)}{X^{T}_{M_{\lambda}[k-1]}}\\
        &=\frac{X^{T}_{M_{\lambda}}X^{T}_{M_{\lambda}[k]}X^{T}_{M_{\lambda}[k-1]}-(X^{T}_{M_{\lambda}[k-1]})^{2}}{X^{T}_{M_{\lambda}[k-1]}},\quad\text{by using relation \eqref{XMlambda} for $t=k$}\\
        &=X^{T}_{M_{\lambda}}X^{T}_{M_{\lambda}[k]}-X^{T}_{M_{\lambda}[k-1]}
    \end{align*}
\end{proof}

\begin{remark}
\label{XMlambdaSecondKind}
    By Lemma~\ref{ChebyshevRecurRelnForMlambda}, the $X^{T}_{M_{\lambda}[k]}$ satisfy the normalized Chebyshev polynomials of second kind.
\end{remark}

For any $M\in\mathrm{obj}(\mathcal{C}(Q))$, we denote by $x^{T}_{M}\coloneqq \restr{X^{T}_{M}}{\lbrace x_{i}=1\rbrace}$, i.e., the integer obtained by specializing $X^{T}_{M}$ at $x_{i}=1$. Using the same notation as above, notice that if $\mathcal{F}_{\mathcal{T}}$ is the infinite frieze pattern coming from a non-homogeneous stable tube $(\mathcal{T},\tau)$, then for $k\in\mathbb{N}$ the $k$-th growth coefficient is $s_{k}(\mathcal
{F}_{\mathcal{T}})=x^{T}_{R_1[kn]}-x^{T}_{R_{2}[kn-2]}$. Recall that we adopted the convention that $R_i[0] \cong 0$, we thus set  $x^{T}_{R_i[0]} = 1$ and we extend this to the negative integers $-1,-2$ by $x^{T}_{R_i[-1]} = 0$ and $x^{T}_{R_{i}[-2]}=-1$.

\begin{theorem}
\label{UnderstandingHigherGrowthCoeff}
    Let \( Q \) be an affine quiver and let \( \mathcal{C}(Q) \) be its associated cluster category. Fix a cluster-tilting object $T\in\mathcal{C}(Q)$. The 
    growth coefficients 
    of the infinite frieze patterns $\F_{\mathcal T}$ of all tubes 
    of $\mathcal{C}(Q)$ are the same. Furthermore, the $k$-th growth coefficient is given by the integer
    \[
        s_k(\mathcal{F}_{\mathcal{T}}) = x^{T}_{M_{\lambda}[k]}-x^{T}_{M_{\lambda}[k-2]},
    \]
    where \( M_{\lambda}[j] \) is the indecomposable module of quasi-length \( j \) in any stable homogeneous tube \( (\mathcal{H}_{\lambda}, \tau) \) of \( \mathcal{C}(Q) \), and \( x^{T}_{M_{\lambda}[k]} \) is the integer obtained by specializing \( X^{T}_{M_{\lambda}[k]} \) at \( x_i = 1 \).

\end{theorem}

\begin{proof} 
    Let $(\mathcal{T}_i,\tau)_{1\le i\le l}$ be the non-homogeneous tubes in $\Gamma(\mathbb{K}Q\mod)$, where $l\le 3$. Denote the corresponding infinite frieze patterns by $\mathcal{F}_{i}$ for $1\leq i \leq l$. 
    Recall that $s_0(\F_i)=2$ (Proposition~\ref{PropertiesofGrowthCoeff}(2)). In the homogeneous tubes, since 
    $x_{M_\lambda[-2]}^T=-1$ (as defined above), 
    we can write $s_0(\F_i)$ as $x^{T}_{M_{\lambda}[0]}-x^{T}_{M_{\lambda}[-2]}=1-(-1)$. 
    Recalling that $x^{T}_{M_{\lambda}[-1]}=0$ and  using Proposition~\ref{InterpretationOfGrowth}, we can write $s_1(\F_i)=x^T_{R_1[n]}-x^T_{R_2[n-2]}$ as 
    \[
    s_{1}(\mathcal{F}_i)=x^{T}_{M_{\lambda}[1]}-x^{T}_{M_{\lambda}[-1]}=x^{T}_{M_{\lambda}}
    \]
    which is independent of the choice of $\lambda$. Therefore, the initial conditions for the growth coefficients of the frieze pattern $\F_i$ and of the frieze pattern of any homogeneous tube are the same. 
    Now by using the relationship between the normalized Chebyshev polynomials of the first and second kinds, along with Remark~\ref{rem:Cheby1-2} and Remark~\ref{XMlambdaSecondKind}, we obtain the identity $s_k(\mathcal{F}_i) = x^{T}_{M_{\lambda}[k]} - x^{T}_{M_{\lambda}[k-2]}$.

    As a consequence, the statement for the 
    homogeneous tubes follows. 
\end{proof}

\begin{remark}
\label{kbraceletsandXMlambda[k]}
    Along with Remark \ref{relatingBraceletsandMlambda} and Remark \ref{XMlambdaSecondKind}, we get that the cluster algebra element $X(\text{Brac}_{k})$ is equal to the difference $X_{M_{\lambda}[k]}-X_{M_{\lambda}[k-2]}$.
\end{remark}

Now we use the above Theorem \ref{UnderstandingHigherGrowthCoeff} to give a representation theoretic interpretation of growth coefficients studied in Section \ref{sec:cl-alg-mod-gr}.

Let $Q$ be an affine quiver of type $\widetilde{A}$, $\widetilde{D}$ or $\widetilde{E}$. 
There exists a mutation sequence from $Q$ to one of the quivers in the Figures~\ref{fig-amn},\,\ref{fig-dn},\,\ref{fig-e}, we denote the resulting quiver 
by $\Gamma$. 
Applying the above mutation sequence starting from $\mathbb{K}Q\cong \oplus_{i=1}^{m}P_i$, where $P_i$ is the indecomposable projective $\mathbb{K}Q$-module at vertex $i$, we get a basic cluster-tilting object $T$ in the cluster category $\mathcal{C}(Q)$. Let $\mathcal{B}=\mathrm{End}_{\mathcal{C}(Q)}(T)$ be the cluster-tilted algebra. It is known that the quiver of $\mathcal{B}$ is equal to $\Gamma$. 
In particular, $\mathcal{B}=\KK\Gamma/J$ for some admissible ideal $J\subseteq\KK\Gamma$. Denote by $\rmod$ the category of finite-dimensional right $\mathcal{B}$-modules. By~\cite{BuanCluster-titlted,KellerReitenCluster-tilted} we know that the functor 
\[
    \varphi : \mathcal{C}(Q)  \rightarrow  \rmod \quad\text{given by} \quad X  \mapsto \mathrm{Hom}_{\mathcal{C}(Q)}(T,X)
\]
induces an equivalence of categories $\mathcal{C}(Q)/\langle\Sigma T\rangle \cong \rmod$, where $\langle\Sigma T\rangle$ denotes the ideal of morphisms of $\mathcal{C}(Q)$ which factor through a direct sum of copies of $\Sigma T$. Let $\Lambda=\KK\Delta$ be the path algebra of the Kronecker subquiver of $\Gamma$ supported on the vertices $1$ and $3$. 
We write $N_{\nu}$ to denote 
the indecomposable $\Lambda$-module below: 
\[
\Delta:\;\;
    \begin{tikzcd}
	   1 \\
	   0
	   \arrow[shift right, from=2-1, to=1-1]
	   \arrow[shift left, from=2-1, to=1-1]
    \end{tikzcd}
    \quad\quad\quad\quad\quad
    N_{\nu}:\;\;
    \begin{tikzcd}
	   {\KK} \\
	   {\KK}
	   \arrow["1"', shift right, from=2-1, to=1-1]
	   \arrow["\nu", shift left, from=2-1, to=1-1]
    \end{tikzcd} 
    \quad\quad\quad\quad\quad
    N_{\infty}:\;\;
    \begin{tikzcd}
	   {\KK} \\
	   {\KK}
	   \arrow["0"', shift right, from=2-1, to=1-1]
	   \arrow["1", shift left, from=2-1, to=1-1]
    \end{tikzcd}  
\]
Since $J$ is an admissible ideal it does not affect the subquiver $\Delta$ and therefore,  
every module in 
$\lbrace N_\nu\rbrace_{\nu\in\mathbb{P}^1}$ is also a $\mathcal{B}$-module.

We recall that a module $M$ is 
\textit{homogeneous} if 
$\tau M\cong M$.

\begin{lemma}
\label{HomModOfB}
The functor $\varphi$ permutes the family 
$\{M_\lambda\}_\lambda$ and almost all $N_\nu$ 
are images of the $M_\lambda$. 
In particular, there are infinitely many $\lambda\in \mathbb{P}^1$ such that 
$\varphi(M_\lambda)\cong N_\nu$ for some 
$\nu=\nu(\lambda)$. 

\end{lemma}

\begin{proof}
Using~\cite[Proposition 3.2]{BuanCluster-titlted} and the structure of the AR quiver for $\rmod$ we get that $\varphi M_{\lambda}$ is a quasi-simple homogeneous $\mathcal{B}$-module 
and so $\varphi$ permutes the homogeneous tubes. 

By~\cite[Theorem D]{CrawleyBoevey1988OnTA}, 
infinitely many of the $N_\nu$ are 
homogeneous. 
The fact that they are quasi-simple follows by using their dimension vector. 

\end{proof}

\begin{prop}
\label{prop:g_rep} 
For almost all $\lambda\in \mathbb{P}^1$, we 
have $X_{M_{\lambda}}^{T}=
    \begin{cases}
        \frac{x_0}{x_1}+\frac{x_1}{x_0}+\frac{x_ax_bx_c}{x_1x_0} & \text{if}\;Q=\widetilde{D}_{m},\widetilde{E}_{6},\widetilde{E}_{7},\widetilde{E}_{8}\\
        \frac{x_0}{x_1}+\frac{x_1}{x_0}+\frac{x_ax_b}{x_1x_0} & \text{if}\;Q=\widetilde{A}_{m}
    \end{cases}$
\end{prop}

\begin{proof}
We will prove the result for $Q=\widetilde{D}_{m},\widetilde{E}_{6},\widetilde{E}_{7},\widetilde{E}_{8}$. The argument for $Q=\widetilde{A}_{m}$ is similar. 
Let $\lambda$ be such that $\varphi(M_\lambda)\cong N_\nu$ for some $\nu$, this exists by Lemma \ref{HomModOfB}. 
For simplicity we will write $N$ instead of $N_{\nu}$. The subrepresentations of $N$ are the zero module $0$, 
the simple module $S_1$, and $N$ itself. 
Moreover, its Euler-Poincar\'e characteristic $\chi(\Gr_{e_i}(\varphi M_{\lambda}))$ is $1$. 
We recall that $\mathrm{Hom}_{\mathcal{B}}(S_{i},S_{j})=0$, for $i\neq j$, and that $\dim\mathrm{Ext}^{1}_{\mathcal{B}}(S_i,S_j)=\#\lbrace\text{arrows from $i$ to $j$ in $\Gamma$}\rbrace$. Hence, we conclude that
    \begin{align*}
        \langle S_i,S_1\rangle_{a}&=\dim\mathrm{Hom}_{\mathcal{B}}(S_i,S_1)-\dim\mathrm{Ext}^{1}_{\mathcal{B}}(S_i,S_1)-\dim\mathrm{Hom}_{\mathcal{B}}(S_1,S_i)+\dim\mathrm{Ext}^{1}_{\mathcal{B}}(S_1,S_i)\\
        &=
        \begin{cases}
            1 & \text{if}\;i=a,b,c\\
            -2 & \text{if}\; i=0\\
            0 & \text{otherwise}
        \end{cases}
    \end{align*}
    
A minimal projective presentation of $N$ is given by
\[
P_{1}^{\mathcal{B}}\rightarrow P_{0}^{\mathcal{B}}\rightarrow N\rightarrow 0
\]
where $P_{1}^{\mathcal{B}},P_{0}^{\mathcal{B}}\in\mathcal{B}\mod$ are indecomposable projective $\mathcal{B}$-modules at the vertices $1$ and $0$, respectively. Since $M_{\lambda}$ is homogeneous in $\mathcal{C}(Q)$, 
we have $\Sigma M_\lambda\cong M_\lambda$ and by~\cite[Lemma~2.1 (2)]{P08}, we conclude that 
\[
-\mathrm{coind}_{\mathcal{B}}\,M_{\lambda}=[P_{0}^{\mathcal{B}}]-[P_{1}^{\mathcal{B}}]
\]

    Notice that the socle $\mathrm{soc}(N)=S_1$ and $\mathrm{top}(N)=S_0$. Hence, we get that $\dim\mathrm{Hom}_{\mathcal{B}}(S_i,N)=0=\dim\mathrm{Hom}_{\mathcal{B}}(N,S_i)$ for $i\neq 1,0$. Since $N$ is a homogeneous $\mathcal{B}$-module, we have by AR duality that
    \begin{align*}
        \dim\mathrm{Ext}^{1}_{\mathcal{B}}(S_i,N)=\dim\underline{\mathrm{Hom}}_{\mathcal{B}}(N,S_i)\quad\text{and}\quad \dim\mathrm{Ext}^{1}_{\mathcal{B}}(N,S_i)=\dim\overline{\mathrm{Hom}}_{\mathcal{B}}(S_i,N)
    \end{align*}
    Additionally, for $l=1,0$, we have
    \begin{align*}
        \dim\mathrm{Ext}^{1}_{\mathcal{B}}(S_l,N)=\dim\mathrm{Ext}^{1}_{\Lambda}(S_l,N)\quad\text{and}\quad \dim\mathrm{Ext}^{1}_{\mathcal{B}}(N,S_l)=\dim\mathrm{Ext}^{1}_{\Lambda}(N,S_l)
    \end{align*}
    Therefore, we get
    \begin{align*}
        \langle S_i,N\rangle_{a}&=\dim\mathrm{Hom}_{\mathcal{B}}(S_i,N)-\dim\mathrm{Ext}^{1}_{\mathcal{B}}(S_i,N)-\dim\mathrm{Hom}_{\mathcal{B}}(N,S_i)+\dim\mathrm{Ext}^{1}_{\mathcal{B}}(N,S_i)\\
        &=
        \begin{cases}
            2 & \text{if}\;i=1\\
            -2 & \text{if}\; i=0\\
            0 & \text{otherwise}
        \end{cases}
    \end{align*}
    Combining all of this, we get 
    \begin{align*}
        X^{T}_{M_{\lambda}}&=\mathbf{x}^{-\mathrm{coind}_{\mathcal{B}}\,M_{\lambda}}\sum_{e}\chi(\Gr_{e}(\varphi  M_{\lambda}))\prod_{i}x_{i}^{\langle S_i,\,e\rangle_{a}}\\
        &=\frac{x_0}{x_1}\left(1+\frac{x_ax_bx_c}{x_0^{2}}+\frac{x_1^{2}}{x_0^{2}}\right)\\
        &=\frac{x_0}{x_1}+\frac{x_1}{x_0}+\frac{x_ax_bx_c}{x_1x_0}
    \end{align*}
as claimed. 
\end{proof}

%
\section{Examples}

In this section we illustrate the main results in this article by calculating the first few growth coefficients for two 
examples.

\begin{example}

Let $Q$ be the following quiver of 
type $\widetilde{D}_4$, 
\[
Q:
\begin{tikzcd}[row sep=small, column sep=small]
    1 && 4 \\
    & 3 \\
    2 && 5
    \arrow[from=1-3, to=2-2]
    \arrow[from=2-2, to=1-1]
    \arrow[from=2-2, to=3-1]
    \arrow[from=3-3, to=2-2]
\end{tikzcd}
\]

The AR quiver $\Gamma(\mathbb{K}Q\mod)$ has three non-homogeneous stable tubes $(\mathcal{T}_1,\tau), (\mathcal{T}_2,\tau), (\mathcal{T}_3,\tau)$, each of rank $2$. For each $1 \leq i \leq 3$, denote the $\tau$-cycle of quasi-simples at the mouth of tube $\mathcal{T}_i$ by $(R^{(i)}_{1}, R_{2}^{(i)})$. The modules $R^{(i)}_j$ are given by the following representations:

\medskip
\begingroup
\renewcommand{\arraystretch}{1.2}
\setlength{\tabcolsep}{5pt}

\noindent
\scriptsize
\begin{tabular}{c|c|c}
\begin{tabular}{>{\centering\arraybackslash}m{0.28\linewidth}}
\begin{tikzcd}[row sep=4pt, column sep=4pt]
    \KK && \KK \\
    & \KK \\
    \KK && \KK
    \arrow["{1}"', from=1-3, to=2-2]
    \arrow["{1}"', from=2-2, to=1-1]
    \arrow["{1}"', from=2-2, to=3-1]
    \arrow["{1}"', from=3-3, to=2-2]
\end{tikzcd} \hspace{4mm}
\begin{tikzcd}[row sep=4pt, column sep=4pt]
    0 && 0 \\
    & \KK \\
    0 && 0
    \arrow[ from=1-3, to=2-2]
    \arrow[ from=2-2, to=1-1]
    \arrow[ from=2-2, to=3-1]
    \arrow[ from=3-3, to=2-2]
\end{tikzcd}\\
 $R_1^{(1)}$ \hspace{15mm} $R_2^{(1)}$
\end{tabular}
&
\begin{tabular}{>{\centering\arraybackslash}m{0.28\linewidth}}
\begin{tikzcd}[row sep=4pt, column sep=4pt]
    \KK && 0 \\
    & \KK \\
    0 && \KK
    \arrow[ from=1-3, to=2-2]
    \arrow["{1}"', from=2-2, to=1-1]
    \arrow[ from=2-2, to=3-1]
    \arrow["{1}"', from=3-3, to=2-2]
\end{tikzcd} \hspace{4mm}
\begin{tikzcd}[row sep=4pt, column sep=4pt]
    0 && \KK \\
    & \KK \\
    \KK && 0
    \arrow["{1}"', from=1-3, to=2-2]
    \arrow[ from=2-2, to=1-1]
    \arrow["{1}"', from=2-2, to=3-1]
    \arrow[ from=3-3, to=2-2]
\end{tikzcd}\\
$R_1^{(2)}$ \hspace{15mm} $R_2^{(2)}$
\end{tabular}
&
\begin{tabular}{>{\centering\arraybackslash}m{0.28\linewidth}}
\begin{tikzcd}[row sep=4pt, column sep=4pt]
    0 && 0 \\
    & \KK \\
    \KK && \KK
    \arrow[ from=1-3, to=2-2]
    \arrow[ from=2-2, to=1-1]
    \arrow["{1}"', from=2-2, to=3-1]
    \arrow["{1}"', from=3-3, to=2-2]
\end{tikzcd} \hspace{4mm}
\begin{tikzcd}[row sep=4pt, column sep=4pt]
    \KK && \KK \\
    & \KK \\
    0 && 0
    \arrow["{1}"', from=1-3, to=2-2]
    \arrow["{1}"', from=2-2, to=1-1]
    \arrow[ from=2-2, to=3-1]
    \arrow[ from=3-3, to=2-2]
\end{tikzcd}\\
$R_1^{(3)}$ \hspace{15mm} $R_2^{(3)}$
\end{tabular}
\end{tabular}

\endgroup
\medskip

In this case, we can compute the CC map by calculating the number of submodules (there are no higher multiplicities in the Euler Poincar\'e characteristics). We get three $2$-periodic infinite frieze patterns $\mathcal{F}_1,\mathcal{F}_2$, and $\mathcal{F}_3$, with quiddity sequences $(8,2), (4,4)$, and $(4,4)$, respectively.

\begingroup
\scriptsize
\[ 
    \mathcal{F}_1:\qquad\quad
    \begin{tikzcd}[row sep=0.5ex, column sep=0.5ex]
	   & \dotsb && 0 && 0 && 0 && 0 & {} & \dotsb \\
	   \dotsb && 1 && 1 && 1 && 1 && 1 && \dotsb \\
	   & \dotsb && 8 && 2 && 8 && 2 && \dotsb \\
	   \dotsb && 15 && 15 && 15 && 15 && 15 && \dotsb \\
	   & \dotsb && 28 && 112 && 28 && 112 && \dotsb \\
	   \dotsb && 209 && 209 && 209 && 209 && 209 \\
	   & \dotsb && 1560 && 390 && 1560 && 390 && \dotsb \\
	   \dotsb && 2911 && 2911 && 2911 && 2911 && 2911 && \dotsb \\
	   &&& \vdots && \vdots && \vdots && \vdots
    \end{tikzcd}
\]
\endgroup

\medskip

\setlength{\tabcolsep}{4pt}      
\renewcommand{\arraystretch}{0.75}

\begin{center}
\scriptsize
\begin{tabular}{@{} >{\centering\arraybackslash}m{0.45\linewidth}
                 !{\vrule width 1pt} 
                 >{\centering\arraybackslash}m{0.45\linewidth} @{}}
\(\mathcal{F}_2:\)
\(
\begin{array}{@{}ccccccccc@{}}
  & \dotsb & & 0 & & 0 & & 0 & \dotsb \\[3pt]
   & & 1 & & 1 & & 1 & &  \\[3pt]
  & \dotsb & & 4 & & 4 & & 4 & \dotsb \\[3pt]
   & & 15 & & 15 & & 15 & &  \\[3pt]
  & \dotsb & & 56 & & 56 & & 56 & \dotsb \\[3pt]
   & & 209 & & 209 & & 209 & &  \\[3pt]
  & & & \vdots & & \vdots & & \vdots &
\end{array}
\)
&
\(\mathcal{F}_3:\)
\(
\begin{array}{@{}ccccccccc@{}}
  & \dotsb & & 0 & & 0 & & 0 & \dotsb \\[3pt]
   & & 1 & & 1 & & 1 & &  \\[3pt]
  & \dotsb & & 4 & & 4 & & 4 & \dotsb \\[3pt]
   & & 15 & & 15 & & 15 & &  \\[3pt]
  & \dotsb & & 56 & & 56 & & 56 & \dotsb \\[3pt]
   & & 209 & & 209 & & 209 & &  \\[3pt]
  & & & \vdots & & \vdots & & \vdots &
\end{array}
\)
\end{tabular}
\end{center}
\medskip

We find 
$s_{1}(\mathcal{F}_1)=s_{1}(\mathcal{F}_2)=s_{1}(\mathcal{F}_3)=14$ 
and 
$s_{2}(\mathcal{F}_1)=s_{2}(\mathcal{F}_2)=s_{2}(\mathcal{F}_3)=194$.
Let us look at the homogeneous tubes: 
For $\lambda\in\KK\setminus\lbrace 0,1\rbrace$, the quasi-simple $M_{\lambda}$ in a homogeneous 
tube $(\mathcal{H}_{\lambda},\tau)$ has the following representation

\[
\begin{tikzcd}[row sep=3ex, column sep=3ex]
	&& {\KK} && {\KK} \\
	{M_{\lambda}:} &&& {\KK^{2}} \\
	&& {\KK} && {\KK}
	\arrow["{\spmat{1 & 1}}", from=2-4, to=1-3]
	\arrow["{\spmat{\lambda & 1}}"',from=2-4, to=3-3] 
	\arrow["{\spmat{1 \\ 0}}"',from=1-5, to=2-4] 
	\arrow["{\spmat{0 \\ 1}}"',from=3-5, to=2-4] 
\end{tikzcd}
\]

We have $\underline{\dim}\,M_{\lambda}=\delta=(1,1,2,1,1)$. Since $M_{\lambda}$ is a general representation of $\rep_{\beta}(Q)$, we can use Schofield's criterion to calculate the Schofield subdimension vectors of $M_{\lambda}$ in a recursive manner. Table \ref{tab:grchi} records the subdimension vectors \( e \), along with the Euler–Poincar\'e characteristics of the associated quiver Grassmannians.

\begin{table}[htbp]
  \centering
  \renewcommand{\arraystretch}{1.2}
  \begin{tabular}{>{\centering\arraybackslash}m{6.5cm} >{\centering\arraybackslash}m{4cm}}
    \toprule
    \textbf{Subdimension vectors \( e \) of \( \delta \)} & \(\chi(\Gr_{e}(M_{\lambda}))\) \\
    \midrule
    \((0, 0, 0, 0, 0)\) & 1 \\
    \((1, 0, 0, 0, 0)\) & 1 \\
    \((0, 1, 0, 0, 0)\) & 1 \\
    \((1, 1, 0, 0, 0)\) & 1 \\
    \((1, 0, 1, 0, 0)\) & 1 \\
    \((0, 1, 1, 0, 0)\) & 1 \\
    \((1, 1, 1, 0, 0)\) & 2 \\
    \((1, 1, 2, 0, 0)\) & 1 \\
    \((1, 1, 1, 1, 0)\) & 1 \\
    \((1, 1, 2, 1, 0)\) & 1 \\
    \((1, 1, 1, 0, 1)\) & 1 \\
    \((1, 1, 2, 0, 1)\) & 1 \\
    \((1, 1, 2, 1, 1)\) & 1 \\
    \bottomrule
  \end{tabular}
  \medskip
  \caption{Euler-Poincar\'e characteristics $\chi(\Gr_e(M_{\lambda}))$ for subdimension vectors $e$}
  \label{tab:grchi}
\end{table}

Hence, we see that $x_{M_{\lambda}}=\restr{X_{M_{\lambda}}}{x_i=1}=14$. 
Using Lemma \ref{ChebyshevRecurRelnForMlambda}, we calculate $x_{M_{\lambda}[2]}=195$, and $x_{M_{\lambda}[3]}=2716$. Hence, we see that 
\begin{align*}
    &s_{1}(\mathcal{F}_i)=x_{M_{\lambda}}\\
    &s_{2}(\mathcal{F}_i)=x_{M_{\lambda}[2]}-x_{M_{\lambda}[0]}\\
    &s_{3}(\mathcal{F}_i)=x_{M_{\lambda}[3]}-x_{M_{\lambda}}
\end{align*}

Moreover, we can use Theorem~\ref{multCK2} to compute the CC-map for \( M_\lambda \): 
 \[
     X_{M_\lambda}(\mathbf{x}) = 
     \frac{1}{x_1 x_2 x_3^2 x_4 x_5}
     \left(
     \begin{aligned}      
        &x_1^2 x_2^2 x_3^2 + 2x_1^2 x_2^2 x_3 + x_1^2 x_2^2 \\
         &+ 4x_1 x_2 x_3 x_4 x_5 + 2x_1 x_2 x_4 x_5 \\
         &+ x_3^2 x_4^2 x_5^2 + 2x_3 x_4^2 x_5^2 + x_4^2 x_5^2
     \end{aligned}
     \right)
 \]

\end{example}

\begin{example}
    
Let $Q$ be the following quiver of 
type $\widetilde{E}_6$.

\[\begin{tikzcd}
	&&& 5 \\
	{Q:} &&& 4 \\
	& 3 & 2 & 1 & 6 & 7
	\arrow[from=1-4, to=2-4]
	\arrow[from=2-4, to=3-4]
	\arrow[ from=3-2, to=3-3]
	\arrow[from=3-3, to=3-4]
	\arrow[ from=3-5, to=3-4]
	\arrow[from=3-6, to=3-5]
\end{tikzcd}\]
The AR quiver $\Gamma(\mathbb{K}Q\mod)$ has three non-homogeneous stable tubes $(\mathcal{T}_1,\tau), (\mathcal{T}_2,\tau), (\mathcal{T}_3,\tau)$, of rank $2,3$, and $3$, respectively. Denote the $\tau$-cycle of quasi-simples at the mouth of tube $\mathcal{T}_1,\mathcal{T}_2$, and $\mathcal{T}_3$ by $(R^{(1)}_{1}, R_{2}^{(1)}),(R^{(2)}_{1}, R_{2}^{(2)},R_{3}^{(2)})$, and $(R^{(3)}_{1}, R_{2}^{(3)},R_{3}^{(3)})$, respectively. The modules $R^{(i)}_j$ are rigid and have the following dimension vectors
\[
    \underline{\dim}\,R^{(1)}_{1}=(1,1,0,1,0,1,0)\qquad \underline{\dim}\,R^{(1)}_{2}=(2,1,1,1,1,1,1)
\]
\[
    \underline{\dim}\,R^{(2)}_{1}=(1,1,1,0,0,1,0)\quad \underline{\dim}\,R^{(2)}_{2}=(1,0,0,1,0,1,1)\quad \underline{\dim}\,R^{(2)}_{3}=(1,1,0,1,1,0,0)
\]
\[
    \underline{\dim}\,R^{(3)}_{1}=(1,1,0,0,0,1,1)\quad \underline{\dim}\,R^{(3)}_{2}=(1,1,1,1,0,0,0)\quad \underline{\dim}\,R^{(3)}_{3}=(1,0,0,1,1,1,0)
\]
For an explicit description of the above representations we refer the readers to~\cite[Chapter 13, p.~166]{SS07}.

Corresponding to each of these three tubes, we get three infinite frieze patterns $\mathcal{F}_1,\mathcal{F}_2$, and $\mathcal{F}_3$ of period $2,3$, and $3$, respectively. We can compute the 
Euler-Poincar\'e characteristics of the quasi-simple rigid modules in these tube and obtain the quiddity sequences $(9,36), (7,7,7)$, and $(7,7,7)$, respectively.

\[  
\mathcal{F}_1:\qquad
\begin{tikzcd}[row sep=0.5ex, column sep=0.5ex]
	   & \dotsb && 0 && 0 && 0 && 0 & {} & \dotsb \\
	   \dotsb && 1 && 1 && 1 && 1 && 1 && \dotsb \\
	   & \dotsb && 9 && 36 && 9 && 36 && \dotsb \\
	   \dotsb && 323 && 323 && 323 && 323 && 323 && \dotsb \\
	   & \dotsb && 11592 && 2898 && 1152 && 2898 && \dotsb \\
	   && \vdots && \vdots && \vdots && \vdots && \vdots
    \end{tikzcd}
\]

\medskip
\medskip
\setlength{\tabcolsep}{4pt}      
\renewcommand{\arraystretch}{0.75}

\begin{center}
\scriptsize
\begin{tabular}{@{} >{\centering\arraybackslash}m{0.45\linewidth}
                 !{\vrule width 1pt} 
                 >{\centering\arraybackslash}m{0.45\linewidth} @{}}
\(\mathcal{F}_2:\)
\(
\begin{array}{@{}ccccccccc@{}}
  & \dotsb & & 0 & & 0 & & 0 & \dotsb \\[3pt]
   & & 1 & & 1 & & 1 & &  \\[3pt]
  & \dotsb & & 7 & & 7 & & 7 & \dotsb \\[3pt]
   & & 48 & & 48 & & 48 & &  \\[3pt]
  & \dotsb & & 329 & & 329 & & 329 & \dotsb \\[3pt]
   & & & \vdots & & \vdots & & \vdots &
\end{array}
\)
&
\(\mathcal{F}_3:\)
\(
\begin{array}{@{}ccccccccc@{}}
  & \dotsb & & 0 & & 0 & & 0 & \dotsb \\[3pt]
   & & 1 & & 1 & & 1 & &  \\[3pt]
  & \dotsb & & 7 & & 7 & & 7 & \dotsb \\[3pt]
   & & 48 & & 48 & & 48 & &  \\[3pt]
  & \dotsb & & 329 & & 329 & & 329 & \dotsb \\[3pt]
   & & & \vdots & & \vdots & & \vdots &
\end{array}
\)
\end{tabular}
\end{center}

\medskip

So we find that $s_{1}(\mathcal{F}_1)=s_{1}(\mathcal{F}_2)=s_{1}(\mathcal{F}_3)=322$.

For $\lambda\in\KK\setminus\lbrace 0,1\rbrace$, the quasi-simple $M_{\lambda}$ in a homogeneous stable tube $(\mathcal{H}_{\lambda},\tau)$ has $\underline{\dim}\,M_{\lambda}=(3,2,1,2,1,2,1)$. For an explicit description of the representation corresponding 
to $M_\lambda$, 
we refer the reader to~\cite[Chapter 13, p.~167]{SS07}. It can be verified that $x_{M_{\lambda}}=322$. We omit the computation here, as the list of subdimension vectors involved is too extensive to include. 
\end{example}

\appendix

\section{Background from 
representation theory}
\label{sec:tame-hered-alg}

Here, we recall facts we use from representation theory. For more details and definitions, we refer to ~\cite{ASS06, SS07, DW17}.

We let $\KK$ be an algebraically closed field. Let $Q$ be an affine quiver and denote by $\mathbb{K}Q$ the path algebra of $Q$. The \emph{dimension vector} $\beta \in(\mathbb{Z}_{\geq 0})^{Q_0}$ of a representation $V \in \rep(Q)$ is defined by $\beta(i) = \dim_{\mathbb {K}} V_i$. The \emph{Euler characteristic} of two representations $V,W$ of $Q$, denoted by $\chi(V,W)$, is defined as
\[
    \chi(V,W) \coloneqq \dim\Hom (V,W) - \dim\Ext (V,W)
\]
Let $\alpha = \underline{\dim}\,V$ and $\beta = \underline{\dim}\,W$. The \emph{Euler form} or \emph{Ringel form} on $\mathbb{R}^{ Q_ 0}$ is defined by
\[
    \langle \alpha, \beta\rangle =\sum_{i\in Q_ 0}\alpha(i)\beta(i) -\sum_{a\in Q _1}\alpha(ta)\beta(ha)
\]
By~\cite[Proposition 2.5.2]{DW17} the Euler characteristic $\chi(V,W)$ is 
equal to $\langle \alpha, \beta\rangle$.

%
\subsection{Auslander-Reiten theory} 

The Auslander-Reiten quiver (AR quiver) of $\mathbb{K}Q$, which we denote by $\Gamma(\mathbb{K}Q\mod)$, provides a combinatorial and visual tool to study its module category. Its vertices correspond to isomorphism classes of indecomposable modules, and its arrows represent irreducible morphisms between them. As every object of $\mathbb{K}Q\mod$ can be uniquely written as a direct sum of indecomposable objects, the AR quiver characterizes the module category.
The AR quiver comes with the so-called Auslander--Reiten translation (AR translation) which is denoted by $\tau$.

The AR quiver $\Gamma(\mathbb{K}Q\mod)$ has infinitely many connected components. There is a unique component containing all indecomposable projective modules, called the \emph{preprojective component} $\mathcal{P}(\mathbb{K}Q)$, and a unique component containing all indecomposable injective modules, called the \emph{preinjective component} $\mathcal{I}(\mathbb{K}Q)$. Any connected component which is neither preprojective nor preinjective is called a \emph{regular component}. We denote by $\mathcal{R}(\mathbb{K}Q)$ the family of all the regular components of $\Gamma(\mathbb{K}Q\mod)$. The homomorphisms in $\mathbb{K}Q\mod$ between 
the different types of components are directed, they only go in one direction: 
\[
    \mathrm{Hom}_{\mathbb{K}Q}(\mathcal{R}(\mathbb{K}Q),\mathcal{P}(\mathbb{K}Q))=\mathrm{Hom}_{\mathbb{K}Q}(\mathcal{I}(\mathbb{K}Q),\mathcal{R}(\mathbb{K}Q))=\mathrm{Hom}_{\mathbb{K}Q}(\mathcal{I}(\mathbb{K}Q),\mathcal{P}(\mathbb{K}Q))=0
\]
An indecomposable module lying in $\mathcal{P}(\mathbb{K}Q)$, $\mathcal{R}(\mathbb{K}Q)$, or $\mathcal{I}(\mathbb{K}Q)$ is called \emph{preprojective}, \emph{regular}, and \emph{preinjective}, respectively. We denote by $\delta$ the unique isotropic Schur root. For $M\in\mathbb{K}Q\mod$, the \emph{defect} of $M$ is defined as $\partial(M)\coloneqq\langle \delta,\underline{\dim}\, M\rangle=-\langle\underline{\dim}\,M,\delta\rangle$. The defect tells us which component 
an indecomposable belongs to: If $M$ is indecomposable, then $M$ is preprojective, regular, or
preinjective if and only if $\partial(M)$ is negative, zero, or positive, respectively~\cite[Lemma 2, pg 26]{WCrawleyLecQuivers}. 
A vertex $e\in Q_{0}$ is called an \emph{extending vertex} if $\delta(e)=1$.

Let $A_\infty$ denote the linear quiver with vertex set $\mathbb{N}$ and arrows $i\to i+1$. For a fixed $r\geq 1$, denote by $\mathbb{Z}A_{\infty}/\langle\tau^{r}\rangle$ the orbit of $\mathbb{Z}A_{\infty}$ under the action of the infinite cyclic group $\langle\tau^{r}\rangle$ generated by $\tau^{r}$. It is known that every regular component $\mathcal{T}$ of $\Gamma(\mathbb{K}Q)$ is isomorphic to a quotient of this form, namely
\[
\mathcal{T}\cong\mathbb{Z}A_\infty/\langle\tau^n\rangle
\]
for a uniquely determined integer $n\ge1$. We will write $(\mathcal{T},\tau)$ for such a component. It is called a \emph{stable tube of rank} $n$. 
The integer $n$ measures the periodicity of the component (i.e. the width of the tube). In case $n=1$, the tube is called \emph{homogeneous}, and for $n>1$ it is called \emph{non-homogeneous}. 
Let $(\mathcal{T},\tau)$ be a stable tube of rank $n\ge1$. 
A sequence $(x_1, \dotsc , x_r)$ of
points of $\mathcal{T}$ is said to be a $\tau$-cycle if $\tau x_1 = x_r, \tau x_2 = x_1, \dotsc , \tau x_r =x_{r-1}$. The set of all points in $\mathcal{T}$ having exactly one immediate predecessor (or, equivalently, exactly one immediate successor) is called the \emph{mouth} of $\mathcal{T}$. A representation $V\in \mathbb{K}Q\mod$ is a {\em brick} or a \emph{Schur} representation if $\mathrm{End}_{\mathbb{K}Q}(V)\cong \KK$. We call two bricks $R$ and $R'$ in $\mathbb{K}Q\mod$ to be \emph{orthogonal} if $\mathrm{Hom}_{\mathbb{K}Q}(R, R') = 0$, and $\mathrm{Hom}_{\mathbb{K}Q}(R', R) = 0$.

Let \( R_1, \dotsc, R_n \) be a family of pairwise orthogonal bricks in the category \( \mathbb{K}Q\mod \). Define an \emph{extension category}
\[
    \xi_{\mathbb{K}Q} = \mathrm{Ext}_{\mathbb{K}Q}(R_1, \dotsc, R_n)
\]
as the full subcategory of \( \mathbb{K}Q\mod \) consisting of all modules \( M \) that admit a finite filtration
\[
    M = M_l \supset M_{l-1} \supset \dotsb \supset M_1\supset M_0 = 0, \quad \text{for some } l \geq 1,
\]
such that each successive quotient \( M_i/M_{i-1} \) is isomorphic to one of the bricks \( R_1, \dotsc, R_n \), for all \( 0 < i \leq l \). In other words, $\xi_{\mathbb{K}Q}$ is the smallest additive subcategory of $\mathbb{K}Q\mod$ containing the bricks $R_i$ and is closed under extensions. It is known that $\xi_{\mathbb{K}Q}$ is an exact abelian subcategory of $\mathbb{K}Q\mod$, and $\lbrace R_1,\dotsc,R_n\rbrace$ is a complete set of pairwise non-isomorphic simple objects in $\xi_{\mathbb{K}Q}$, also called \emph{quasi-simples}. The module $M_{1}$ is called the \emph{quasi-socle} of $M$, denoted by $\mathrm{q.soc}(M)$, and $M/M_{l-1}$ is called the \emph{quasi-top} of $M$, denoted by $\mathrm{q.top}(M)$. An object $M$ in the category $\xi_{\mathbb{K}Q}$ is defined to be
\emph{uniserial}, if $N_1 \subseteq N_2$ or $N_2 \subseteq N_1$, for each pair $N_1, N_2$ of sub-objects of $M$ in $\xi_{\mathbb{K}Q}$. The length of the chain of sub-objects of a uniserial object $M$ in the
category $\xi_{\mathbb{K}Q}$ is called the \emph{quasi-length} of $M$.

In fact, every indecomposable object $M$ of the category $\xi_{\mathbb{K}Q}$ is uniserial and is
of the form $M\cong R_i[j]$, where $i \in \lbrace 1, \dotsc, n\rbrace$ and $j \geq 1$. Moreover, $\xi_{\mathbb{K}Q}$ forms a a stable tube $(\mathcal{T}_{\xi_{\mathbb{K}Q}},\tau)$ of rank $n$  in $\mathcal{R}(\mathbb{K}Q)$ with $\lbrace R_1,\dotsc,R_n\rbrace$ as a complete set of modules at the mouth of the tube $\mathcal{T}_{\xi_{\mathbb{K}Q}}$. Conversely, every stable tube $(\mathcal{T},\tau)$ in $\mathcal{R}(\mathbb{K}Q)$ is given by such an extension category. 

Given \( i \in \{1, \dotsc, n\} \) and \( j \geq 1 \), we have \( \mathrm{End}(R_i[j]) \cong \KK \) if and only if \( j \leq n \). The modules $R_{i}[j]$ are rigid, i.e., \( \mathrm{Ext}^1_{\mathbb{K}Q}(R_i[j], R_i[j])  = 0 \), if and only if \( j \leq n - 1 \). Additionally, if the stable tube \( \mathcal{T} \) is homogeneous, then
$\mathrm{Ext}^1_{\mathbb{K}Q}(M, M) \neq 0$
for any indecomposable \( M \in \mathcal{T} \).

The regular components form an orthogonal family of stable tubes parametrized by the projective line $\mathbb{P}^1$, with at most three non-homogeneous stable tubes.

\begin{theorem}~\cite[Section 9]{WCrawleyLecQuivers}
\label{ParameterizationOfTubes}
    Let $e$ be an extending vertex, and $P_{e}$ be the indecomposable projective at the vertex $e$ with $p=\underline{\dim}\, P_{e}$. There exists a unique indecomposable preprojective module $E$ (up to isomorphism) of dimension vector $\delta+p$. The map $\theta\mapsto\mathrm{coker}\,\theta$ gives a bijection between $\mathbb{P}\mathrm{Hom}_{\mathbb{K}Q}(P_{e},E)\rightarrow \Omega$, where
    \[
        \Omega = \left\{ M \;\middle|\;M\,\text{indecomposable},\,\underline{\dim}\, M= \delta \;\text{and}\; d(e) \neq 0,\; \text{where } d = \underline{\dim}\,\mathrm{q.top}(M) \right\}
    \]
    Moreover, each stable tube contains a unique module in the set $\Omega$.
\end{theorem}

\subsection{Cluster Categories}

Write $\db(\mathbb{K}Q)$ for the bounded derived category of $\mathbb{K}Q\mod$, equipped with the shift functor $\Sigma$ and the Auslander-Reiten translation $\tau$. We have an autoequivalence of $\db(\mathbb{K}Q)$ given by $\tau^{-1}\Sigma$. The \emph{cluster category} associated with $Q$ is the orbit category
\[
    \mathcal{C}(Q) \;=\; \db(\mathbb{K}Q)/\tau^{-1}\Sigma,
\]
introduced in~\cite{BMRRT06}. The category $\mathcal{C}(Q)$ is a  Krull-Schmidt, 2-Calabi-Yau, triangulated category~\cite{BMRRT06, Keller2005}.

\begin{definition}
    A cluster character on $\mathcal{C}(Q)$ with values 
    in a commutative ring $\mathcal{R}$ is a map 
    \[
        X_{?} : \mathrm{ obj}(\mathcal{C}(Q)) \rightarrow \mathcal{R}
    \]
    such that for any $M,N\in\mathrm{obj}(\mathcal{C}(Q))$, we have
    \begin{enumerate}
        \item if $M\cong N$, then $X_{M} = X_{N}$
        \item we have $X_{M\oplus N} = X_{M}X_{N}$
        \item if $\mathrm{Ext}^{1}_{\mathcal{C}(Q)}(M,N) \cong \KK$, then $X_MX_N=X_B+X_{B'}$, where $B$ and $B'$ are the unique (up to isomorphism) objects in $\mathcal{C}(Q)$ such that there exists non-split triangles $$M \rightarrow B \rightarrow N \rightarrow \Sigma M \textrm{ and } N \rightarrow B'\rightarrow M \rightarrow \Sigma N.$$
    \end{enumerate}
\end{definition}

Fix a basic cluster-tilting object $T\in\mathcal{C}(Q)$, i.e., $\mathrm{Hom}_{\mathcal{C}(Q)}(T,\Sigma T) = 0$ and for any $X \in \mathrm{obj}(\mathcal{C}(Q))$, if $\mathrm{Hom}_{\mathcal{C}(Q)}(X,\Sigma T) = 0$, then $X$ belongs to the full subcategory $\mathsf{add}\,T$ formed by the direct summands of sums of copies of $T$. Let $\mathcal{B}\coloneqq\mathrm{End}_{\mathcal{C}(Q)}(T)$. The algebra $\mathcal{B}$ is called a \emph{cluster-tilted algebra of type $Q$}, and it is given by a Jacobi-finite quiver with potential $(\Gamma,W)$; for more details see~\cite[Section 3]{Amiot2009}. Denote by $\rmod$ the category of finite-dimensional right $\mathcal{B}$-modules. By~\cite{BuanCluster-titlted,KellerReitenCluster-tilted} we know that the functor 
\[
    \varphi : \mathcal{C}(Q)  \rightarrow  \rmod \quad\text{given by} \quad X  \mapsto \mathrm{Hom}_{\mathcal{C}(Q)}(T,X)
\]
induces an equivalence of categories $\mathcal{C}(Q)/\langle\Sigma T\rangle \cong \rmod$, where $\langle\Sigma T\rangle$ denotes the ideal of morphisms of $\mathcal{C}(Q)$ which factor through a direct sum of copies of $\Sigma T$.

Let $M\in\mathrm{obj}(\mathcal{C}(Q))$. By~\cite{KellerReitenCluster-tilted} there exists a triangle
\[
    M_1 \to M_0 \to M \to \Sigma M_1
\]
with $M_0,M_1\in\mathsf{add}\,T$. The \emph{index of $M$}, denoted by 
$\mathrm{ind}_{B}\,M$, is defined to be the class $[\varphi M_0] - [\varphi M_1]$ of the Grothendieck group $\mathrm{K}_{0}(\mathsf{proj}\,\mathcal{B})$. Similarly, define the \emph{coindex of $M$}, denoted by $\mathrm{coind}_{B}\,M$, to be the class $[\varphi M'_0] - [\varphi M'_1]$ of the Grothendieck group $\mathrm{K}_{0}(\mathsf{proj}\,\mathcal{B})$, where 
\[
    M \to \Sigma^2 M'_0 \to \Sigma^2 M'_1 \to \Sigma M
\]
is a triangle in $\mathcal{C}(Q)$ with $M'_0,M'_1\in\mathsf{add}\,T$.

Following~\cite{P08}, denote by $\mathrm{K}_{0}^{sp}(\rmod)$ the `split' Grothendieck group of $\rmod$, i.e., 
the quotient of the free abelian group on the set of isomorphism classes $[N]$ of
finite-dimensional $B$-modules $N$, modulo the subgroup generated by all elements
$$
    [N_1\oplus N_2] -[N_1] -[N_2]
$$
We have a bilinear form 
\[
    \langle - ,- \rangle : \mathrm{K}_{0}^{sp}(\rmod) \times \mathrm{K}_{0}^{sp}(\rmod) \rightarrow \mathbb{Z}
\]
defined by $\langle [N],[N']\rangle  = \mathrm{Hom}_{B}(N,N')-\mathrm{Ext}^{1}_{B}(N,N')$ for all finite-dimensional $B$-modules $N$ and $N'$. Define an antisymmetric bilinear form on $ \mathrm{K}_{0}^{sp}(\rmod)$ by setting 
$$
\langle [N],[N']\rangle_a \;=\; \langle N,N'\rangle - \langle N',N\rangle 
$$
for all finite-dimensional $B$-modules $N$ and $N'$.

\begin{lemma}[see Lemma 1.3~\cite{P08}]
\label{lem:lf}
    For any $i=1,\ldots, n$, the linear form 
    $\langle S_i,?\rangle_a \;:\, \mathrm{K}_{0}^{sp}(\rmod) \to \mathbb{Z}$ 
    induces a well-defined form
    $$
    \langle S_i,?\rangle_a \;:\, \mathrm{K}_{0}(\rmod) \to \mathbb{Z}.
    $$
\end{lemma}

For $N\in \rmod$ and $e=(e(i))_{i\in \Gamma_0}\in\mathrm{K}_{0}(\rmod)$ the Grothendieck group of $\rmod$, we write
\[
    \Gr_e(N)=\big\{(U_i)_{i\in \Gamma_0}\ \big|\ U_i\subseteq N_i,\ \dim U_i=e(i),\ \text{for all}\, a:i\to j,\ N_a(U_i)\subseteq U_j\big\}
\]
for the \emph{quiver Grassmannian} of submodules of $N$, and denote its \emph{Euler-Poincar\'e characteristic} by $\chi(\Gr_{e}(N))$. Denote by $T_1,\dotsc, T_m$ the pairwise non-isomorphic indecomposable direct summands of $T$. For $1\leq i\leq m$, let $S_i$ be the top of the projective $B$-module $P_i=\varphi\,T_i$. The set $\{S_i, i = 1,\ldots,m\}$ is a set of representatives for the iso-classes of simple $B$-modules. We will write $\mathbf{x}^\alpha$ for $\prod_{i=1}^n x_i^{[\alpha : P_i]}$ where $e \in\mathrm{K}_{0}(\mathsf{proj}\,\mathcal{B})$ and $[\alpha:P_i]$ is the $i$-th coefficient of $\alpha$ in the basis $[P_1],\dotsc,[P_m]$. Similar to~\cite{CC06,CK06}, we have the following Caldero-Chapoton map (CC map for short).

\begin{thmdef}[see Theorem 1.4~\cite{P08}]
\label{multCK2}
    The map $X^{T}_? : \mathrm{obj}(\mathcal
    {C}(Q))\rightarrow \mathbb{Q}[x_i^{\pm 1},1\leq i \leq m]$ defined by
    \[
        X^T_M \;=\;
        \begin{cases}
            x_i & \text{ if}\;M \simeq \Sigma T_i, \\
            \displaystyle
            \sum_{e} \chi\!\big(\Gr_{e}(\varphi M)\big)\;
            \prod_{i=1}^m x_i^{\langle S_i,e\rangle_a-\langle S_i,\varphi M\rangle}
            &\text{ otherwise}
        \end{cases}
    \]
    is a cluster character. In the special case where $T=\KK Q$, we will simply write $X_{?}$ for the CC-map. 
\end{thmdef}

By~\cite[Corollary 3.11]{Amiot2009}, the cluster category $\mathcal{C}(\mathcal{B})$ of the the cluster-tilted algebra $\mathcal{B}$ is triangle equivalent to $\mathcal{C}(Q)$. Hence, the CC-map for $\mathcal{C}(\mathcal{B})$ can be identified with the CC-map in the above Theorem \ref{multCK2}.

\bibliographystyle{alpha}
\bibliography{references}

\end{document}